\documentclass[12pt]{article}
\title{\bf Solving Equations Using Khovanskii Bases}
\author{Barbara Betti, Marta Panizzut and Simon Telen}
\date{}

\usepackage[utf8]{inputenc}
\usepackage[T1]{fontenc}
\usepackage[top=3.5cm,bottom=3.5cm,left=3cm,right=3cm]{geometry} 
\usepackage{blindtext}
\usepackage{amsmath}
\usepackage{amsthm}
\usepackage{amssymb}
\usepackage{booktabs}
\usepackage{color}
\usepackage{kbordermatrix}
\usepackage{graphicx}
\graphicspath{{./images/}}
\usepackage[dvipsnames]{xcolor}
\usepackage{tikz}
\usepackage{tikz-cd}
\usepackage{tikz,pgfplots}
\usepackage{bm}
\usepackage{bbm}
\usepackage{nicefrac}
\usepackage{mathtools}
\usepackage{epstopdf}
\usepackage{multicol}
\usepackage{rotating}
\usepackage[shortcuts]{extdash}
\usepackage{hyperref}
\usepackage{tikz,pgf}
\usetikzlibrary{arrows}
\usepackage{xspace}
\usepackage[all]{xy}
\usepackage{mathrsfs}
\usepackage{verbatim}
\usepackage{marvosym}
\usepackage{alphabeta}

\usepackage[greek,english]{babel}

\usepackage[draft]{minted}
\setminted[julia]{frame=lines,
rulecolor=\color{white!80!black},
fontsize=\small,
numbers=right,
numbersep=-5pt,
obeytabs=true,
encoding = utf8,
tabsize=4}

\makeatletter
\def\PYGtango@reset{\let\PYGtango@it=\relax \let\PYGtango@bf=\relax%
    \let\PYGtango@ul=\relax \let\PYGtango@tc=\relax%
    \let\PYGtango@bc=\relax \let\PYGtango@ff=\relax}
\def\PYGtango@tok#1{\csname PYGtango@tok@#1\endcsname}
\def\PYGtango@toks#1+{\ifx\relax#1\empty\else%
    \PYGtango@tok{#1}\expandafter\PYGtango@toks\fi}
\def\PYGtango@do#1{\PYGtango@bc{\PYGtango@tc{\PYGtango@ul{%
    \PYGtango@it{\PYGtango@bf{\PYGtango@ff{#1}}}}}}}
\def\PYGtango#1#2{\PYGtango@reset\PYGtango@toks#1+\relax+\PYGtango@do{#2}}

\expandafter\def\csname PYGtango@tok@w\endcsname{\let\PYGtango@ul=\underline\def\PYGtango@tc##1{\textcolor[rgb]{0.97,0.97,0.97}{##1}}}
\expandafter\def\csname PYGtango@tok@err\endcsname{\def\PYGtango@tc##1{\textcolor[rgb]{0.64,0.00,0.00}{##1}}\def\PYGtango@bc##1{\setlength{\fboxsep}{0pt}\fcolorbox[rgb]{0.94,0.16,0.16}{1,1,1}{\strut ##1}}}
\expandafter\def\csname PYGtango@tok@x\endcsname{\def\PYGtango@tc##1{\textcolor[rgb]{0.00,0.00,0.00}{##1}}}
\expandafter\def\csname PYGtango@tok@c\endcsname{\let\PYGtango@it=\textit\def\PYGtango@tc##1{\textcolor[rgb]{0.56,0.35,0.01}{##1}}}
\expandafter\def\csname PYGtango@tok@cm\endcsname{\let\PYGtango@it=\textit\def\PYGtango@tc##1{\textcolor[rgb]{0.56,0.35,0.01}{##1}}}
\expandafter\def\csname PYGtango@tok@cp\endcsname{\let\PYGtango@it=\textit\def\PYGtango@tc##1{\textcolor[rgb]{0.56,0.35,0.01}{##1}}}
\expandafter\def\csname PYGtango@tok@c1\endcsname{\let\PYGtango@it=\textit\def\PYGtango@tc##1{\textcolor[rgb]{0.56,0.35,0.01}{##1}}}
\expandafter\def\csname PYGtango@tok@cs\endcsname{\let\PYGtango@it=\textit\def\PYGtango@tc##1{\textcolor[rgb]{0.56,0.35,0.01}{##1}}}
\expandafter\def\csname PYGtango@tok@k\endcsname{\let\PYGtango@bf=\textbf\def\PYGtango@tc##1{\textcolor[rgb]{0.13,0.29,0.53}{##1}}}
\expandafter\def\csname PYGtango@tok@kc\endcsname{\let\PYGtango@bf=\textbf\def\PYGtango@tc##1{\textcolor[rgb]{0.13,0.29,0.53}{##1}}}
\expandafter\def\csname PYGtango@tok@kd\endcsname{\let\PYGtango@bf=\textbf\def\PYGtango@tc##1{\textcolor[rgb]{0.13,0.29,0.53}{##1}}}
\expandafter\def\csname PYGtango@tok@kn\endcsname{\let\PYGtango@bf=\textbf\def\PYGtango@tc##1{\textcolor[rgb]{0.13,0.29,0.53}{##1}}}
\expandafter\def\csname PYGtango@tok@kp\endcsname{\let\PYGtango@bf=\textbf\def\PYGtango@tc##1{\textcolor[rgb]{0.13,0.29,0.53}{##1}}}
\expandafter\def\csname PYGtango@tok@kr\endcsname{\let\PYGtango@bf=\textbf\def\PYGtango@tc##1{\textcolor[rgb]{0.13,0.29,0.53}{##1}}}
\expandafter\def\csname PYGtango@tok@kt\endcsname{\let\PYGtango@bf=\textbf\def\PYGtango@tc##1{\textcolor[rgb]{0.13,0.29,0.53}{##1}}}
\expandafter\def\csname PYGtango@tok@o\endcsname{\let\PYGtango@bf=\textbf\def\PYGtango@tc##1{\textcolor[rgb]{0.81,0.36,0.00}{##1}}}
\expandafter\def\csname PYGtango@tok@ow\endcsname{\let\PYGtango@bf=\textbf\def\PYGtango@tc##1{\textcolor[rgb]{0.13,0.29,0.53}{##1}}}
\expandafter\def\csname PYGtango@tok@p\endcsname{\let\PYGtango@bf=\textbf\def\PYGtango@tc##1{\textcolor[rgb]{0.00,0.00,0.00}{##1}}}
\expandafter\def\csname PYGtango@tok@n\endcsname{\def\PYGtango@tc##1{\textcolor[rgb]{0.00,0.00,0.00}{##1}}}
\expandafter\def\csname PYGtango@tok@na\endcsname{\def\PYGtango@tc##1{\textcolor[rgb]{0.77,0.63,0.00}{##1}}}
\expandafter\def\csname PYGtango@tok@nb\endcsname{\def\PYGtango@tc##1{\textcolor[rgb]{0.13,0.29,0.53}{##1}}}
\expandafter\def\csname PYGtango@tok@bp\endcsname{\def\PYGtango@tc##1{\textcolor[rgb]{0.20,0.40,0.64}{##1}}}
\expandafter\def\csname PYGtango@tok@nc\endcsname{\def\PYGtango@tc##1{\textcolor[rgb]{0.00,0.00,0.00}{##1}}}
\expandafter\def\csname PYGtango@tok@no\endcsname{\def\PYGtango@tc##1{\textcolor[rgb]{0.00,0.00,0.00}{##1}}}
\expandafter\def\csname PYGtango@tok@nd\endcsname{\let\PYGtango@bf=\textbf\def\PYGtango@tc##1{\textcolor[rgb]{0.36,0.21,0.80}{##1}}}
\expandafter\def\csname PYGtango@tok@ni\endcsname{\def\PYGtango@tc##1{\textcolor[rgb]{0.81,0.36,0.00}{##1}}}
\expandafter\def\csname PYGtango@tok@ne\endcsname{\let\PYGtango@bf=\textbf\def\PYGtango@tc##1{\textcolor[rgb]{0.80,0.00,0.00}{##1}}}
\expandafter\def\csname PYGtango@tok@nf\endcsname{\def\PYGtango@tc##1{\textcolor[rgb]{0.00,0.00,0.00}{##1}}}
\expandafter\def\csname PYGtango@tok@py\endcsname{\def\PYGtango@tc##1{\textcolor[rgb]{0.00,0.00,0.00}{##1}}}
\expandafter\def\csname PYGtango@tok@nl\endcsname{\def\PYGtango@tc##1{\textcolor[rgb]{0.96,0.47,0.00}{##1}}}
\expandafter\def\csname PYGtango@tok@nn\endcsname{\def\PYGtango@tc##1{\textcolor[rgb]{0.00,0.00,0.00}{##1}}}
\expandafter\def\csname PYGtango@tok@nx\endcsname{\def\PYGtango@tc##1{\textcolor[rgb]{0.00,0.00,0.00}{##1}}}
\expandafter\def\csname PYGtango@tok@nt\endcsname{\let\PYGtango@bf=\textbf\def\PYGtango@tc##1{\textcolor[rgb]{0.13,0.29,0.53}{##1}}}
\expandafter\def\csname PYGtango@tok@nv\endcsname{\def\PYGtango@tc##1{\textcolor[rgb]{0.00,0.00,0.00}{##1}}}
\expandafter\def\csname PYGtango@tok@vc\endcsname{\def\PYGtango@tc##1{\textcolor[rgb]{0.00,0.00,0.00}{##1}}}
\expandafter\def\csname PYGtango@tok@vg\endcsname{\def\PYGtango@tc##1{\textcolor[rgb]{0.00,0.00,0.00}{##1}}}
\expandafter\def\csname PYGtango@tok@vi\endcsname{\def\PYGtango@tc##1{\textcolor[rgb]{0.00,0.00,0.00}{##1}}}
\expandafter\def\csname PYGtango@tok@l\endcsname{\def\PYGtango@tc##1{\textcolor[rgb]{0.00,0.00,0.00}{##1}}}
\expandafter\def\csname PYGtango@tok@m\endcsname{\let\PYGtango@bf=\textbf\def\PYGtango@tc##1{\textcolor[rgb]{0.00,0.00,0.81}{##1}}}
\expandafter\def\csname PYGtango@tok@mf\endcsname{\let\PYGtango@bf=\textbf\def\PYGtango@tc##1{\textcolor[rgb]{0.00,0.00,0.81}{##1}}}
\expandafter\def\csname PYGtango@tok@mh\endcsname{\let\PYGtango@bf=\textbf\def\PYGtango@tc##1{\textcolor[rgb]{0.00,0.00,0.81}{##1}}}
\expandafter\def\csname PYGtango@tok@mi\endcsname{\let\PYGtango@bf=\textbf\def\PYGtango@tc##1{\textcolor[rgb]{0.00,0.00,0.81}{##1}}}
\expandafter\def\csname PYGtango@tok@il\endcsname{\let\PYGtango@bf=\textbf\def\PYGtango@tc##1{\textcolor[rgb]{0.00,0.00,0.81}{##1}}}
\expandafter\def\csname PYGtango@tok@mo\endcsname{\let\PYGtango@bf=\textbf\def\PYGtango@tc##1{\textcolor[rgb]{0.00,0.00,0.81}{##1}}}
\expandafter\def\csname PYGtango@tok@ld\endcsname{\def\PYGtango@tc##1{\textcolor[rgb]{0.00,0.00,0.00}{##1}}}
\expandafter\def\csname PYGtango@tok@s\endcsname{\def\PYGtango@tc##1{\textcolor[rgb]{0.31,0.60,0.02}{##1}}}
\expandafter\def\csname PYGtango@tok@sb\endcsname{\def\PYGtango@tc##1{\textcolor[rgb]{0.31,0.60,0.02}{##1}}}
\expandafter\def\csname PYGtango@tok@sc\endcsname{\def\PYGtango@tc##1{\textcolor[rgb]{0.31,0.60,0.02}{##1}}}
\expandafter\def\csname PYGtango@tok@sd\endcsname{\let\PYGtango@it=\textit\def\PYGtango@tc##1{\textcolor[rgb]{0.56,0.35,0.01}{##1}}}
\expandafter\def\csname PYGtango@tok@s2\endcsname{\def\PYGtango@tc##1{\textcolor[rgb]{0.31,0.60,0.02}{##1}}}
\expandafter\def\csname PYGtango@tok@se\endcsname{\def\PYGtango@tc##1{\textcolor[rgb]{0.31,0.60,0.02}{##1}}}
\expandafter\def\csname PYGtango@tok@sh\endcsname{\def\PYGtango@tc##1{\textcolor[rgb]{0.31,0.60,0.02}{##1}}}
\expandafter\def\csname PYGtango@tok@si\endcsname{\def\PYGtango@tc##1{\textcolor[rgb]{0.31,0.60,0.02}{##1}}}
\expandafter\def\csname PYGtango@tok@sx\endcsname{\def\PYGtango@tc##1{\textcolor[rgb]{0.31,0.60,0.02}{##1}}}
\expandafter\def\csname PYGtango@tok@sr\endcsname{\def\PYGtango@tc##1{\textcolor[rgb]{0.31,0.60,0.02}{##1}}}
\expandafter\def\csname PYGtango@tok@s1\endcsname{\def\PYGtango@tc##1{\textcolor[rgb]{0.31,0.60,0.02}{##1}}}
\expandafter\def\csname PYGtango@tok@ss\endcsname{\def\PYGtango@tc##1{\textcolor[rgb]{0.31,0.60,0.02}{##1}}}
\expandafter\def\csname PYGtango@tok@g\endcsname{\def\PYGtango@tc##1{\textcolor[rgb]{0.00,0.00,0.00}{##1}}}
\expandafter\def\csname PYGtango@tok@gd\endcsname{\def\PYGtango@tc##1{\textcolor[rgb]{0.64,0.00,0.00}{##1}}}
\expandafter\def\csname PYGtango@tok@ge\endcsname{\let\PYGtango@it=\textit\def\PYGtango@tc##1{\textcolor[rgb]{0.00,0.00,0.00}{##1}}}
\expandafter\def\csname PYGtango@tok@gr\endcsname{\def\PYGtango@tc##1{\textcolor[rgb]{0.94,0.16,0.16}{##1}}}
\expandafter\def\csname PYGtango@tok@gh\endcsname{\let\PYGtango@bf=\textbf\def\PYGtango@tc##1{\textcolor[rgb]{0.00,0.00,0.50}{##1}}}
\expandafter\def\csname PYGtango@tok@gi\endcsname{\def\PYGtango@tc##1{\textcolor[rgb]{0.00,0.63,0.00}{##1}}}
\expandafter\def\csname PYGtango@tok@go\endcsname{\let\PYGtango@it=\textit\def\PYGtango@tc##1{\textcolor[rgb]{0.00,0.00,0.00}{##1}}}
\expandafter\def\csname PYGtango@tok@gp\endcsname{\def\PYGtango@tc##1{\textcolor[rgb]{0.56,0.35,0.01}{##1}}}
\expandafter\def\csname PYGtango@tok@gs\endcsname{\let\PYGtango@bf=\textbf\def\PYGtango@tc##1{\textcolor[rgb]{0.00,0.00,0.00}{##1}}}
\expandafter\def\csname PYGtango@tok@gu\endcsname{\let\PYGtango@bf=\textbf\def\PYGtango@tc##1{\textcolor[rgb]{0.50,0.00,0.50}{##1}}}
\expandafter\def\csname PYGtango@tok@gt\endcsname{\let\PYGtango@bf=\textbf\def\PYGtango@tc##1{\textcolor[rgb]{0.64,0.00,0.00}{##1}}}
\expandafter\def\csname PYGtango@tok@fm\endcsname{\def\PYGtango@tc##1{\textcolor[rgb]{0.00,0.00,0.00}{##1}}}
\expandafter\def\csname PYGtango@tok@vm\endcsname{\def\PYGtango@tc##1{\textcolor[rgb]{0.00,0.00,0.00}{##1}}}
\expandafter\def\csname PYGtango@tok@sa\endcsname{\def\PYGtango@tc##1{\textcolor[rgb]{0.31,0.60,0.02}{##1}}}
\expandafter\def\csname PYGtango@tok@dl\endcsname{\def\PYGtango@tc##1{\textcolor[rgb]{0.31,0.60,0.02}{##1}}}
\expandafter\def\csname PYGtango@tok@mb\endcsname{\let\PYGtango@bf=\textbf\def\PYGtango@tc##1{\textcolor[rgb]{0.00,0.00,0.81}{##1}}}
\expandafter\def\csname PYGtango@tok@ch\endcsname{\let\PYGtango@it=\textit\def\PYGtango@tc##1{\textcolor[rgb]{0.56,0.35,0.01}{##1}}}
\expandafter\def\csname PYGtango@tok@cpf\endcsname{\let\PYGtango@it=\textit\def\PYGtango@tc##1{\textcolor[rgb]{0.56,0.35,0.01}{##1}}}

\makeatother

\makeatletter
\def\PYG@reset{\let\PYG@it=\relax \let\PYG@bf=\relax%
    \let\PYG@ul=\relax \let\PYG@tc=\relax%
    \let\PYG@bc=\relax \let\PYG@ff=\relax}
\def\PYG@tok#1{\csname PYG@tok@#1\endcsname}
\def\PYG@toks#1+{\ifx\relax#1\empty\else%
    \PYG@tok{#1}\expandafter\PYG@toks\fi}
\def\PYG@do#1{\PYG@bc{\PYG@tc{\PYG@ul{%
    \PYG@it{\PYG@bf{\PYG@ff{#1}}}}}}}
\def\PYG#1#2{\PYG@reset\PYG@toks#1+\relax+\PYG@do{#2}}

\expandafter\def\csname PYG@tok@w\endcsname{\def\PYG@tc##1{\textcolor[rgb]{0.73,0.73,0.73}{##1}}}
\expandafter\def\csname PYG@tok@c\endcsname{\let\PYG@it=\textit\def\PYG@tc##1{\textcolor[rgb]{0.25,0.50,0.50}{##1}}}
\expandafter\def\csname PYG@tok@cp\endcsname{\def\PYG@tc##1{\textcolor[rgb]{0.74,0.48,0.00}{##1}}}
\expandafter\def\csname PYG@tok@k\endcsname{\let\PYG@bf=\textbf\def\PYG@tc##1{\textcolor[rgb]{0.00,0.50,0.00}{##1}}}
\expandafter\def\csname PYG@tok@kp\endcsname{\def\PYG@tc##1{\textcolor[rgb]{0.00,0.50,0.00}{##1}}}
\expandafter\def\csname PYG@tok@kt\endcsname{\def\PYG@tc##1{\textcolor[rgb]{0.69,0.00,0.25}{##1}}}
\expandafter\def\csname PYG@tok@o\endcsname{\def\PYG@tc##1{\textcolor[rgb]{0.40,0.40,0.40}{##1}}}
\expandafter\def\csname PYG@tok@ow\endcsname{\let\PYG@bf=\textbf\def\PYG@tc##1{\textcolor[rgb]{0.67,0.13,1.00}{##1}}}
\expandafter\def\csname PYG@tok@nb\endcsname{\def\PYG@tc##1{\textcolor[rgb]{0.00,0.50,0.00}{##1}}}
\expandafter\def\csname PYG@tok@nf\endcsname{\def\PYG@tc##1{\textcolor[rgb]{0.00,0.00,1.00}{##1}}}
\expandafter\def\csname PYG@tok@nc\endcsname{\let\PYG@bf=\textbf\def\PYG@tc##1{\textcolor[rgb]{0.00,0.00,1.00}{##1}}}
\expandafter\def\csname PYG@tok@nn\endcsname{\let\PYG@bf=\textbf\def\PYG@tc##1{\textcolor[rgb]{0.00,0.00,1.00}{##1}}}
\expandafter\def\csname PYG@tok@ne\endcsname{\let\PYG@bf=\textbf\def\PYG@tc##1{\textcolor[rgb]{0.82,0.25,0.23}{##1}}}
\expandafter\def\csname PYG@tok@nv\endcsname{\def\PYG@tc##1{\textcolor[rgb]{0.10,0.09,0.49}{##1}}}
\expandafter\def\csname PYG@tok@no\endcsname{\def\PYG@tc##1{\textcolor[rgb]{0.53,0.00,0.00}{##1}}}
\expandafter\def\csname PYG@tok@nl\endcsname{\def\PYG@tc##1{\textcolor[rgb]{0.63,0.63,0.00}{##1}}}
\expandafter\def\csname PYG@tok@ni\endcsname{\let\PYG@bf=\textbf\def\PYG@tc##1{\textcolor[rgb]{0.60,0.60,0.60}{##1}}}
\expandafter\def\csname PYG@tok@na\endcsname{\def\PYG@tc##1{\textcolor[rgb]{0.49,0.56,0.16}{##1}}}
\expandafter\def\csname PYG@tok@nt\endcsname{\let\PYG@bf=\textbf\def\PYG@tc##1{\textcolor[rgb]{0.00,0.50,0.00}{##1}}}
\expandafter\def\csname PYG@tok@nd\endcsname{\def\PYG@tc##1{\textcolor[rgb]{0.67,0.13,1.00}{##1}}}
\expandafter\def\csname PYG@tok@s\endcsname{\def\PYG@tc##1{\textcolor[rgb]{0.73,0.13,0.13}{##1}}}
\expandafter\def\csname PYG@tok@sd\endcsname{\let\PYG@it=\textit\def\PYG@tc##1{\textcolor[rgb]{0.73,0.13,0.13}{##1}}}
\expandafter\def\csname PYG@tok@si\endcsname{\let\PYG@bf=\textbf\def\PYG@tc##1{\textcolor[rgb]{0.73,0.40,0.53}{##1}}}
\expandafter\def\csname PYG@tok@se\endcsname{\let\PYG@bf=\textbf\def\PYG@tc##1{\textcolor[rgb]{0.73,0.40,0.13}{##1}}}
\expandafter\def\csname PYG@tok@sr\endcsname{\def\PYG@tc##1{\textcolor[rgb]{0.73,0.40,0.53}{##1}}}
\expandafter\def\csname PYG@tok@ss\endcsname{\def\PYG@tc##1{\textcolor[rgb]{0.10,0.09,0.49}{##1}}}
\expandafter\def\csname PYG@tok@sx\endcsname{\def\PYG@tc##1{\textcolor[rgb]{0.00,0.50,0.00}{##1}}}
\expandafter\def\csname PYG@tok@m\endcsname{\def\PYG@tc##1{\textcolor[rgb]{0.40,0.40,0.40}{##1}}}
\expandafter\def\csname PYG@tok@gh\endcsname{\let\PYG@bf=\textbf\def\PYG@tc##1{\textcolor[rgb]{0.00,0.00,0.50}{##1}}}
\expandafter\def\csname PYG@tok@gu\endcsname{\let\PYG@bf=\textbf\def\PYG@tc##1{\textcolor[rgb]{0.50,0.00,0.50}{##1}}}
\expandafter\def\csname PYG@tok@gd\endcsname{\def\PYG@tc##1{\textcolor[rgb]{0.63,0.00,0.00}{##1}}}
\expandafter\def\csname PYG@tok@gi\endcsname{\def\PYG@tc##1{\textcolor[rgb]{0.00,0.63,0.00}{##1}}}
\expandafter\def\csname PYG@tok@gr\endcsname{\def\PYG@tc##1{\textcolor[rgb]{1.00,0.00,0.00}{##1}}}
\expandafter\def\csname PYG@tok@ge\endcsname{\let\PYG@it=\textit}
\expandafter\def\csname PYG@tok@gs\endcsname{\let\PYG@bf=\textbf}
\expandafter\def\csname PYG@tok@gp\endcsname{\let\PYG@bf=\textbf\def\PYG@tc##1{\textcolor[rgb]{0.00,0.00,0.50}{##1}}}
\expandafter\def\csname PYG@tok@go\endcsname{\def\PYG@tc##1{\textcolor[rgb]{0.53,0.53,0.53}{##1}}}
\expandafter\def\csname PYG@tok@gt\endcsname{\def\PYG@tc##1{\textcolor[rgb]{0.00,0.27,0.87}{##1}}}
\expandafter\def\csname PYG@tok@err\endcsname{\def\PYG@bc##1{\setlength{\fboxsep}{0pt}\fcolorbox[rgb]{1.00,0.00,0.00}{1,1,1}{\strut ##1}}}
\expandafter\def\csname PYG@tok@kc\endcsname{\let\PYG@bf=\textbf\def\PYG@tc##1{\textcolor[rgb]{0.00,0.50,0.00}{##1}}}
\expandafter\def\csname PYG@tok@kd\endcsname{\let\PYG@bf=\textbf\def\PYG@tc##1{\textcolor[rgb]{0.00,0.50,0.00}{##1}}}
\expandafter\def\csname PYG@tok@kn\endcsname{\let\PYG@bf=\textbf\def\PYG@tc##1{\textcolor[rgb]{0.00,0.50,0.00}{##1}}}
\expandafter\def\csname PYG@tok@kr\endcsname{\let\PYG@bf=\textbf\def\PYG@tc##1{\textcolor[rgb]{0.00,0.50,0.00}{##1}}}
\expandafter\def\csname PYG@tok@bp\endcsname{\def\PYG@tc##1{\textcolor[rgb]{0.00,0.50,0.00}{##1}}}
\expandafter\def\csname PYG@tok@fm\endcsname{\def\PYG@tc##1{\textcolor[rgb]{0.00,0.00,1.00}{##1}}}
\expandafter\def\csname PYG@tok@vc\endcsname{\def\PYG@tc##1{\textcolor[rgb]{0.10,0.09,0.49}{##1}}}
\expandafter\def\csname PYG@tok@vg\endcsname{\def\PYG@tc##1{\textcolor[rgb]{0.10,0.09,0.49}{##1}}}
\expandafter\def\csname PYG@tok@vi\endcsname{\def\PYG@tc##1{\textcolor[rgb]{0.10,0.09,0.49}{##1}}}
\expandafter\def\csname PYG@tok@vm\endcsname{\def\PYG@tc##1{\textcolor[rgb]{0.10,0.09,0.49}{##1}}}
\expandafter\def\csname PYG@tok@sa\endcsname{\def\PYG@tc##1{\textcolor[rgb]{0.73,0.13,0.13}{##1}}}
\expandafter\def\csname PYG@tok@sb\endcsname{\def\PYG@tc##1{\textcolor[rgb]{0.73,0.13,0.13}{##1}}}
\expandafter\def\csname PYG@tok@sc\endcsname{\def\PYG@tc##1{\textcolor[rgb]{0.73,0.13,0.13}{##1}}}
\expandafter\def\csname PYG@tok@dl\endcsname{\def\PYG@tc##1{\textcolor[rgb]{0.73,0.13,0.13}{##1}}}
\expandafter\def\csname PYG@tok@s2\endcsname{\def\PYG@tc##1{\textcolor[rgb]{0.73,0.13,0.13}{##1}}}
\expandafter\def\csname PYG@tok@sh\endcsname{\def\PYG@tc##1{\textcolor[rgb]{0.73,0.13,0.13}{##1}}}
\expandafter\def\csname PYG@tok@s1\endcsname{\def\PYG@tc##1{\textcolor[rgb]{0.73,0.13,0.13}{##1}}}
\expandafter\def\csname PYG@tok@mb\endcsname{\def\PYG@tc##1{\textcolor[rgb]{0.40,0.40,0.40}{##1}}}
\expandafter\def\csname PYG@tok@mf\endcsname{\def\PYG@tc##1{\textcolor[rgb]{0.40,0.40,0.40}{##1}}}
\expandafter\def\csname PYG@tok@mh\endcsname{\def\PYG@tc##1{\textcolor[rgb]{0.40,0.40,0.40}{##1}}}
\expandafter\def\csname PYG@tok@mi\endcsname{\def\PYG@tc##1{\textcolor[rgb]{0.40,0.40,0.40}{##1}}}
\expandafter\def\csname PYG@tok@il\endcsname{\def\PYG@tc##1{\textcolor[rgb]{0.40,0.40,0.40}{##1}}}
\expandafter\def\csname PYG@tok@mo\endcsname{\def\PYG@tc##1{\textcolor[rgb]{0.40,0.40,0.40}{##1}}}
\expandafter\def\csname PYG@tok@ch\endcsname{\let\PYG@it=\textit\def\PYG@tc##1{\textcolor[rgb]{0.25,0.50,0.50}{##1}}}
\expandafter\def\csname PYG@tok@cm\endcsname{\let\PYG@it=\textit\def\PYG@tc##1{\textcolor[rgb]{0.25,0.50,0.50}{##1}}}
\expandafter\def\csname PYG@tok@cpf\endcsname{\let\PYG@it=\textit\def\PYG@tc##1{\textcolor[rgb]{0.25,0.50,0.50}{##1}}}
\expandafter\def\csname PYG@tok@c1\endcsname{\let\PYG@it=\textit\def\PYG@tc##1{\textcolor[rgb]{0.25,0.50,0.50}{##1}}}
\expandafter\def\csname PYG@tok@cs\endcsname{\let\PYG@it=\textit\def\PYG@tc##1{\textcolor[rgb]{0.25,0.50,0.50}{##1}}}

\makeatother

\usepackage{listings}
\usepackage{hyperref}
\hypersetup{
    pdfpagemode=UseNone,
    colorlinks=true,
    linkcolor=blue,
    citecolor=blue,
    urlcolor=blue,
    allbordercolors=white,
    pdftitle={Eigenvalue methods},
    pdfauthor={Barbara Betti}}
    
\newtheorem{thm}{Theorem}[section]
\newtheorem{lemm}[thm]{Lemma}
\newtheorem{prop}[thm]{Proposition}
\newtheorem{cor}[thm]{Corollary}

\newtheorem{assum}[thm]{Assumption}
\theoremstyle{definition}
\newtheorem{examplex}[thm]{Example} 
\newenvironment{example}
  {\pushQED{\qed}\examplex}
  {\popQED\endexamplex}

\newtheorem{defn}[thm]{Definition}

\theoremstyle{remark}
\newtheorem{remark}{Remark}

\newcommand{\numberset}{\mathbb}

\newcommand{\C}{\mathbb{C}}

\newcommand{\N}{\numberset{N}}

\newcommand{\proj}[1]{\mathbb{P}^{#1}}

\definecolor{cof}{RGB}{219,144,71}
\definecolor{pur}{RGB}{186,146,162}
\definecolor{greeo}{RGB}{91,173,69}
\definecolor{greet}{RGB}{52,111,72}

\def\initial{{\mathrm{in}}}
\def\inv{^{-1}}

\begin{document}

\maketitle

\begin{abstract}
    We develop a new eigenvalue method for solving structured polynomial equations over any field. The equations are defined on a projective algebraic variety which admits a rational parameterization by a Khovanskii basis, e.g., a Grassmannian in its Pl\"ucker embedding. This generalizes established algorithms for toric varieties, and introduces the effective use of Khovanskii bases in computer algebra. We investigate regularity questions and discuss several applications. 
\end{abstract}

\section{Introduction} \label{sec:1}
Let $K$ be a field with algebraic closure $\overline{K}$. We consider the problem of finding all 
\begin{equation} \label{eq:system}
    t \,= \, (t_1, \ldots, t_n) \in \overline{K}^{\,n} \quad \text{such that} \quad f_1(t) \, = \, \cdots \, = \, f_s(t) \, = \, 0, 
\end{equation} 
where $f_1, \ldots, f_s \in K[t_1, \ldots, t_n]$ are polynomials with the following structure. Let $\phi_0, \ldots, \phi_\ell \in K[t_1,\ldots, t_n]$ be a different list of polynomials and let $d_1, \ldots, d_s \in \mathbb{N} \setminus \{0\}$ be positive integers. Our polynomials $f_i$ in \eqref{eq:system} are homogeneous polynomials in the $\phi_j$: 
\begin{equation} \label{eq:fi}
    f_i(t) \, = \, \sum_{|\alpha| = d_i} \,  c_{i,\alpha} \, \phi_0(t)^{\alpha_0} \phi_1(t)^{\alpha_1} \cdots \phi_\ell(t)^{\alpha_\ell}, \quad i = 1, \ldots, s. 
\end{equation} 
Here $\alpha = (\alpha_0, \alpha_1, \ldots, \alpha_\ell) \in \mathbb{N}^{\ell + 1}$ is a tuple of exponents and the sum is over all such tuples with entry sum $|\alpha| = \alpha_0 + \cdots + \alpha_\ell = d_i$. We assume that $c_{i,\alpha} \in K$ are generic in $\overline{K}$. We use the special structure in \eqref{eq:fi} to reformulate \eqref{eq:system} as follows: find all  
\begin{equation} \label{eq:homsystem}
    x \in X \subset \mathbb{P}_{\overline{K}}^\ell, \quad \text{such that} \quad F_1(x) \, = \, \cdots \, = \, F_s(x) \, = \, 0,  \quad \text{with } F_i \, = \, \sum_{|\alpha| = d_i} c_{i,\alpha} \, x^\alpha. 
\end{equation} 
Here $x^\alpha$ is short for $x_0^{\alpha_0}x_1^{\alpha_1} \cdots x_\ell^{\alpha_\ell}$, and $X$ is the unirational variety obtained from
\[ \phi: \, \overline{K}^{\, n} \dashrightarrow \mathbb{P}^\ell_{\overline{K}}, \quad t \, \mapsto \, (\phi_0(t) : \cdots : \phi_\ell(t) ) \]
by taking the closure of its image. We restrict ourselves to the case where the functions $\phi_j$ form a \emph{Khovanskii basis} for the algebra they generate. We elaborate on that assumption below. Any solution $t$ of \eqref{eq:system} gives a solution $x = \phi(t)$ of \eqref{eq:homsystem}. Conversely, if $\phi$ is injective and well-defined on $U \subset \overline{K}^{\,n}$, then every solution $x \in {\rm im}\, \phi \subset X$ of \eqref{eq:homsystem} gives a solution $t = \phi^{-1}(x) \in U$ of \eqref{eq:system}. In this paper, we solve \eqref{eq:homsystem}. 

Let us focus on the case where $n = s$. When $\ell = n = s$ and $\phi_0 = 1, \phi_1 = t_1, \ldots, \phi_n = t_n$, we are dealing with a general system of $n$ polynomial equations of degree $d_1, \ldots, d_n$. Equation \eqref{eq:system} regards them as equations on $\overline{K}^{\, n}$, while \eqref{eq:homsystem} interprets them as equations on $X = \mathbb{P}_{\overline K}^n$. B\'ezout's theorem predicts $d_1 \cdots d_n$ solutions. When $d_1 = \cdots = d_n = 1$ and $\phi_j = t^{\beta_j}$ are monomials, $X$ is a projective toric variety. In this case, Kushnirenko's theorem says that the expected number of solutions is the normalized volume of the convex hull of the exponents $\beta_j$, viewed as points in $\mathbb{R}^n$ \cite[Section 3.4]{telen2022introduction}. It is customary to refer to these standard cases as the \emph{dense} case and the \emph{(unmixed) sparse/toric} case, respectively. Obviously, for any choice of the parameterizing functions $\phi_j$, one can simply expand in monomials and re-interpret the equations \eqref{eq:fi} as being dense (or toric). However, the number of solutions might be far less than predicted by B\'ezout (or Kushnirenko), because the coefficients after expansion are no longer \emph{generic}. Methods for solving dense/toric equations will waste computational effort on trying to compute these `missing' solutions or, in the worst case, they will not work at all. It is therefore crucial to \emph{not} expand in monomials. We keep the structure \eqref{eq:fi}, and seek to exploit it. 
\begin{example}[$n = s = 2, \ell = 4$] \label{ex:duffing}
The \emph{Duffing equation} in physics models damped and driven oscillators. It gives rise to a system $f_1 = f_2 = 0$ in two unknowns, with
\[f_1 \, = \, c_{1,0} + c_{1,1} \, t_1 + c_{1,2} \, t_2 + c_{1,3} \, t_1(t_1^2+t_2^2), \quad f_2 \, = \, c_{2,0} + c_{2,1} \, t_1 + c_{2,2} \, t_2 + c_{2,4} \, t_2(t_1^2+t_2^2). \]
This appears in \cite[Section 3]{breiding2022algebraic}.
B\'ezout's theorem predicts nine solutions of $f_1 = f_2 = 0$ in $\mathbb{P}^2_{\overline{K}}$. However, by \cite[Theorem 3.1]{breiding2022algebraic}, only five of those lie in $\overline{K}^{\, 2}$. Kushnirenko's theorem counts solutions on the projective toric surface parameterized by the seven monomials appearing in $f_1, f_2$. This surface lives in $\mathbb{P}^6_{\overline{K}}$ and has degree nine, which again overcounts the solutions in ${\overline{K}}^{\,2}$ by four. In this paper, we solve \eqref{eq:homsystem}, i.e.~we compute all
\begin{equation} \label{eq:homsysduffing}
x \in X \subset \mathbb{P}^4_{\overline{K}} , \quad \text{such that} \quad F_1(x) \, =  \, F_2(x) \, = \, 0,  \quad \text{with } F_i \, = \, \textstyle \sum_{j = 0}^4 c_{i,j} \, x_j. 
\end{equation}
Here $c_{1,4} = c_{2,3} = 0$. The variety $X$ is the surface in $\mathbb{P}^4_{\overline{K}}$ parameterized by $\phi_0 = 1, \phi_1 = t_1, \phi_2 = t_2, \phi_3 = t_1(t_1^2 + t_2^2), \phi_4 = t_2(t_1^2 + t_2^2)$. Its degree is 5, which agrees with \cite[Theorem 3.1]{breiding2022algebraic}. The three defining equations are shown in Example \ref{ex:duffing2}.
\end{example}
The assumption that $\phi_0, \ldots, \phi_\ell$ form a Khovanskii basis generalizes the toric case: monomials always form a Khovanskii basis. However, Khovanskii bases are much more flexible. Examples in which this Khovanskii assumption is satisfied include Schubert problems \cite{huber1998numerical}, equations expressed in terms of elementary symmetric polynomials \cite[p.~99]{sturmfelsgroebner} and the Duffing oscillators from Example \ref{ex:duffing} \cite{breiding2022algebraic}. The results in the latter paper were recently extended to more general oscillator problems in \cite{Viktoriia}, which also uses Khovanskii bases. We will recall the basics of Khovanskii bases in Section \ref{sec:2}. 

\paragraph{Related work.} 
The toric case, i.e.~the case in which all $\phi_j$ are monomials, is classical. Prime examples are the theory of sparse resultants \cite[Chapter 7]{cox2006using} and the polyhedral homotopy method from \cite{huber1995polyhedral} (which needs $K = \mathbb{C}$). On the homotopy continuation side, the authors of \cite{huber1998numerical} considered the special case where the $\phi_j$ parameterize the Grassmannian in its Pl\"ucker embedding. Recently, Burr, Sottile and Walker developed more general continuation methods based on Khovanskii bases \cite{burr2020numerical}. Their algorithm generalizes the SAGBI homotopy from \cite{huber1998numerical}. Specialized homotopies for the case of oscillators are found in \cite{breiding2022algebraic}. On the algebraic side, the paper \cite{buse2000generalized} goes beyond toric varieties: it develops the theory of resultants for $n + 1$ equations on an $n$-dimensional unirational variety. This leads to approaches for solving square instances of \eqref{eq:homsystem} (i.e.~$n = s$) via B\'ezoutian matrices. Resultants on more general varieties are studied in \cite{monin2020overdetermined}. To the best of our knowledge, Khovanskii bases have not been used before in this context. 

\paragraph{Contributions.} As discussed above, Khovanskii bases have been successfully used in \emph{homotopy solvers}, which work inherently over $K = \mathbb{C}$ and are most natural to use when $n = s$. Our contribution is to optimally exploit structure of the form \eqref{eq:fi} in \emph{computer algebra methods}, which work over any field and for any number of equations. Existing normal form methods, e.g.~Gr\"obner bases, necessarily work with monomials. State-of-the-art implementations use linear algebra operations on a \emph{Macaulay matrix}, whose rows (or columns) are indexed by monomials. We introduce a generalization of such matrices, called \emph{Khovanskii-Macaulay matrices}, in which monomials are replaced by basis elements of the graded pieces $K[X]_d$ of the coordinate ring $K[X]$ of $X$. From a Khovanskii-Macaulay matrix, one can compute multiplication operators modulo the ideal $I$ generated by the $F_i$ from \eqref{eq:homsystem}. Their eigenvalues reveal the solutions (Theorem \ref{thm:eigenvaluethm2}). 
The size of the Khovanskii-Macaulay matrix depends on the \emph{regularity} of the ideal $I$. Important in this regard is the following theorem, which is of independent interest. 
\begin{thm} \label{thm:intro}
    Let $X \subset \mathbb{P}^\ell_{\overline{K}}$ 
    be arithmetically Cohen-Macaulay of dimension $n$, with Hilbert series ${\rm HS}_X(u) = (c_a u^a + c_{a+1}u^{a+1} + \cdots + c_b u^b)/(1-u)^{n+1}$, $c_b \neq 0$. Let $I = \langle F_1, \ldots, F_n \rangle \subset K[X]$ be a homogeneous ideal with ${\rm deg}(F_i) = d_i$, such that $\dim V_X(I) = 0$. Then for all degrees $d \geq \sum_{i=1}^n d_i + b - n$, we have $\dim_K (K[X]/I)_d = {\rm deg} V_X(I)$.
\end{thm}
\noindent In this theorem, $V_X(I)$ is the subscheme of $X$ defined by $I$. Note that no assumptions related to Khovanskii bases are necessary: $X$ does not even need to be unirational. 

Our eigenvalue method for solving \eqref{eq:homsystem} is implemented in a Julia package, called \texttt{KhovanskiiSolving.jl}: \url{https://mathrepo.mis.mpg.de/KhovanskiiSolving}. 

\paragraph{Outline.} Section \ref{sec:2} discusses unirational varieties obtained from Khovanskii bases.  Section \ref{sec:3} introduces Khovanskii-Macaulay matrices, and Section \ref{sec:4} states an eigenvalue theorem (Theorem \ref{thm:eigenvaluethm2}) used for solving \eqref{eq:homsystem}. Section \ref{sec:5} is about regularity when $n = s$. It contains a proof of Theorem \ref{thm:intro}. In Section \ref{sec:6}, we demonstrate our method via our Julia code. Finally, Section \ref{sec:7} solves non-trivial problems on Grassmannians.  

\section{Unirational varieties and Khovanskii bases} \label{sec:2}

A unirational variety $X \subset \mathbb{P}^\ell_{\overline{K}}$ is obtained as the closure of the image of a rational map 
\[ \phi: \, {\overline{K}}^{\,n} \dashrightarrow \mathbb{P}^\ell_{\overline{K}}, \quad t \, \mapsto \, (\phi_0(t) : \cdots : \phi_\ell(t)). \]
We assume that $\phi_j \in K[t_1, \ldots, t_n]$ have coefficients in $K \subseteq \overline{K}$. The coordinate ring of the $K$-variety $X\cap \mathbb{P}_K^\ell$ is $R = K[t_0\phi_0,\dots,t_0\phi_\ell]\subset K[t_0,\dots, t_n]$. Indeed, the sequence
\[ 0 \longrightarrow I(X) \longrightarrow K[x_0, \ldots, x_\ell] \longrightarrow R \longrightarrow 0.\]
is exact. Here $I(X) \subset K[x_0, \ldots, x_\ell]$ is the vanishing ideal of $X\cap \mathbb{P}_K^\ell$, and the second map sends $x_i$ to $t_0 \phi_i$. We will write $K[X] = K[x_0,\ldots, x_\ell]/I(X) \simeq R$. The grading
\begin{equation} \label{eq:gradingR} R \, = \, \bigoplus_{d = 0}^\infty R_d, \quad \text{with} \quad R_d \, = \, \{ t_0^d\cdot p(\phi_0,\dots,\phi_\ell) \mid p\in K[x_0,\dots,x_\ell]_d  \} 
\end{equation}
makes this an isomorphism of graded rings. That is, as $K$-vector spaces, $K[X]_d \simeq R_d$.

In examples, our field is $K = \mathbb{Q}$ or $K = \mathbb{F}_p$. As mentioned in the Introduction, we assume in this article that $\phi_0, \ldots, \phi_\ell$ form a \emph{Khovanskii basis}. This section makes this precise, and it recalls the basic facts about Khovanskii bases exploited in our methods.

The subalgebra $R \subset K[t_0, \ldots, t_n]$ is generated by $\ell + 1$ elements $t_0 \phi_0, \ldots, t_0\phi_\ell$. Such a finite set of generators of an algebra $R$ is called a \emph{Khovanskii basis} if it satisfies some nice properties, as detailed in Definition \ref{K.B. def}.
In a sense, Khovanskii bases are to subalgebras of a polynomial ring what Gröbner bases are to ideals. We should point out that the term Khovanskii basis is used in much more general settings, see \cite{kaveh2012newton,kaveh2019khovanskii}. The instances used in this paper are called \emph{SAGBI bases} (Subalgebra Analog to Gröbner Bases for Ideals) by Robbiano and Sweedler \cite{robbiano2006subalgebra} and \emph{canonical bases} by Sturmfels \cite{sturmfelsgroebner}.

Let $\prec$ be a monomial order on $K[t_0, \ldots, t_n]$. For a subalgebra $R \subset K[t_0, \ldots, t_n]$, the \emph{initial algebra} ${\rm in}_\prec(R)$ is defined as 
\[ {\rm in}_\prec(R) \, = \, K[ \, {\rm in}_\prec(\psi) \, : \, \psi \in R \, ]. \]
A subset ${\cal F} \subset K[t_0, \ldots, t_n]$ generates the subalgebra $R = K[{\cal F}] \subset K[t_0, \ldots, t_n]$. The set of $\prec$-leading terms of ${\cal F}$ is denoted ${\rm in}_\prec({\cal F})$. The inclusion $K[{\rm in}_\prec({\cal F})] \subseteq {\rm in}_\prec(K[{\cal F}])$ is obvious. The opposite inclusion holds precisely when ${\cal F}$ is a \emph{Khovanskii basis} for~$K[{\cal F}]$: 
\begin{defn} \label{K.B. def}
Let $R\subseteq K[t_0,\dots,t_n]$ be a subalgebra of the polynomial ring and $\mathcal{F}\subseteq R$ a set of polynomials. Let $\prec$ be a monomial order on $K[t_0,\dots,t_n]$. 
We say that $\mathcal{F}$ is a \emph{Khovanskii basis} for $R$ with respect to $\prec$ if $\initial_{\prec}{(R)}=K[\initial_{\prec}{(f)} : f\in\mathcal{F}]$.
\end{defn}

For $\omega \in \mathbb{R}^{n+1}$, we define the \emph{$\omega$-initial form} ${\rm in}_\omega(\psi)$ of $\psi$ to be the sum of all terms $c_\alpha t^\alpha$ of $\psi$ for which $w \cdot \alpha$ is minimal. For a finite subset ${\cal F} = \{ \psi_0, \ldots, \psi_\ell \} \subset K[t_0, \ldots, t_n]$, let $\omega \in \mathbb{R}^{n+1}$ be such that the leading terms of ${\cal F}$ with respect to $\prec$ agree with the $\omega$-initial forms: ${\rm in}_\omega(\psi_i) = {\rm in}_{\prec}(\psi_i)$. Such a weight vector $\omega$ is said to \emph{represent} $\prec$ for ${\cal F}$.

\begin{example} \label{KB example}
Let $\prec$ be the degree lexicographic term order on $K[t_0,t_1,t_2]$. On 
\[ \mathcal{F} \, = \, \{ t_0(t_1-t_2), \, t_0(t_2^2-t_2), \, t_0(t_1t_2-t_2), \, t_0(t_1^2-t_2), \, t_0(t_1t_2^2-t_2), \, t_0(t_1^2t_2-t_2) \}, \]
this is represented by the weight vector $\omega = (0,-2,-1)$.
The polynomials in ${\cal F}$ form a Khovanskii basis with respect to ${\prec}$. 
The initial algebra of $R = K[\mathcal{F}]$ is the monomial algebra ${\rm in}_\prec(R) = K[t_0t_1,\, t_0t_2^2,\, t_0t_1t_2,\, t_0t_1^2,\, t_0t_1t_2^2,\, t_0t_1^2t_2]$. In our setup, we write ${\cal F} = \{\psi_0, \ldots, \psi_5\}$ with $\psi_i = t_0 \phi_i$, and the $\phi_i$ parameterize a surface $X$ in $\mathbb{P}^5_{\overline{K}}$. This surface is a degree 5 del Pezzo surface, obtained by blowing up $\mathbb{P}^2_{\overline{K}}$ in the points $(1:0:0),\, (0:1:0),\, (0:0:1)$ and $(1:1:1)$.
Its ideal $I(X)$ is generated by five quadratic polynomials and one cubic polynomial.
\end{example}

Definition \ref{K.B. def} does not require Khovanskii bases to be finite. In contrast with Gr\"obner bases, not every finitely generated subalgebra has a finite Khovanskii basis. A classical example is the invariant ring of the alternating group $A_n$ \cite[Example 11.2]{sturmfelsgroebner}. However, in many practical cases a finite Khovanskii basis exists. 

There is an algorithm to check whether a set $\mathcal{F}=\{\psi_0,\ldots, \psi_\ell \}$ is a Khovanskii basis for $R=K[\mathcal{F}]$ with respect to $\prec$. As above, $\omega \in \mathbb{R}^{n+1}$ represents $\prec$ on ${\cal F}$. Let $A$ be the $(n+1) \times (\ell+1)$ matrix whose columns are the vectors $\alpha_i\in \N^{n+1}$ occuring as leading exponents of the $\psi_i$, i.e., $\initial_\omega{(\psi_i)}=t^{\alpha_i}$. Consider the following two ring maps:
\[ \begin{aligned}
 K[x_0,\dots, x_\ell]&\longrightarrow K[t_1,\dots,t_n], \ &K[x_0,\dots,x_\ell]&\longrightarrow {\rm in}_\prec(R) \\
x_i&\longmapsto \psi_i & x_i&\longmapsto \initial_\omega{(\psi_i)}.
\end{aligned}\]
Their kernels are denoted by $I$ and $I_A$ respectively . The second ideal is the toric ideal associated to the matrix $A$.
The following is {\cite[Theorem~11.4]{sturmfelsgroebner}}.
\begin{thm} \label{thm:khovcheck}
Let $\prec$ be a monomial order on $K[t_0,\dots,t_n]$, represented by $\omega \in \mathbb{R}^{n+1}$ on ${\cal F}$. The set $\mathcal{F}$ is a Khovanskii basis of $R = K[{\cal F}]$ w.r.t.~$\prec$ if and only if $\initial_{A^T\omega}{(I)}=I_A$.
\end{thm}
\begin{example} \label{ex:Khovcheck}
Let ${\cal F} = \{ \psi_0, \ldots, \psi_5 \} \subset K[t_0, t_1, t_2]$ be as in Example \ref{KB example}. The following snippet of \texttt{Macaulay2} code uses Theorem \ref{thm:khovcheck} to test whether ${\cal F}$ is a Khovanskii basis:
\begin{minted}{julia}
loadPackage "QuasiDegrees"
S = QQ[t0,t1,t2, MonomialOrder => {Weights => {0,2,1}}, Global => true];
F = {t0*(t1-t2), ..., t0*(t1^2*t2-t2)}; inF = apply(F,leadTerm);
exps = apply(toList(0..length(F)-1), i->(exponents leadMonomial inF_i)_0);
A = transpose matrix(exps); v = flatten entries transpose(matrix{{0,2,1}}*A);  
R = QQ[x_0..x_5, MonomialOrder => {Weights => v}, Global => true];
ideal(leadTerm(1,ker map(S,R,F))) == toricIdeal(A,R) --true
\end{minted}
\normalsize
Notice that $\mathtt{Weights} = -\omega$ to be consistent with \texttt{Macaulay2} conventions. This check can also be done using \texttt{isSagbi} from the package \texttt{SubalgebraBases} \cite{burr2023subalgebrabases}. 
\end{example}
We switch back to the setting where $\psi_i = t_0 \phi_i$ and $R = K[{\cal F}] = K[t_0 \phi_0, \ldots, t_0\phi_\ell]$. Our matrix construction in Section \ref{sec:3} relies on the knowledge of a $K$-basis for some graded pieces $R_d$ of $R$. The Khovanskii basis property helps to find such bases. 

The grading on $R$ is as in \eqref{eq:gradingR}. It is inherited by the monomial algebras $K[{\rm in}_\prec({\cal F})] \subseteq {\rm in}_\prec(R)$. We write ${\rm HF}_R : \mathbb{Z} \rightarrow \mathbb{N}$ for the Hilbert function of a $\mathbb{Z}$-graded $K$-algebra $R$: ${\rm HF}_R(d) = \dim_K R_d$. As above, $A = [\alpha_0~ \cdots ~\alpha_\ell]$ is the matrix of leading exponents, and 
\[ d \cdot A = \{ \alpha_{i_1} + \cdots + \alpha_{i_d} \, : \, 0 \leq i_j \leq \ell \} \]
consists of all $d$-element sums of the columns of $A$.

\begin{prop} \label{prop:basisforRd}
Suppose ${\cal F} = \{ t_0 \phi_0, \ldots, t_0 \phi_\ell \} \subset K[t_0, \ldots, t_n]$ is a Khovanskii basis for $R = K[{\cal F}]$ with respect to some term order $\prec$.
Then ${\rm HF}_{R}(d) = {\rm HF}_{K[{\rm in}_\prec({\cal F})]}(d)$ for all $d \in \mathbb{Z}$. Moreover, a $K$-basis of $R_d$ is given by
\begin{equation} \label{eq:bdbeta}
b_{\beta} \, = \, t_0^d \cdot \phi_{i_1}\cdots \phi_{i_d}, \quad \beta \in d \cdot A, 
\end{equation}
where $0 \leq i_1 \leq \cdots \leq i_d \leq \ell$ are integers, $\alpha_{i_j}$ is column $i_j$ of $A$ and $\alpha_{i_1} + \cdots + {\alpha_{i_d}} = \beta$.
\end{prop}

Notice that, in Proposition \ref{prop:basisforRd}, the number of basis elements equals the number of elements in $d \cdot A$. This is the number of monomials in $K[{\rm in}_\prec({\cal F})]$ of degree $d$. For a given $\beta$, the set of indices $0 \leq i_1 \leq \cdots \leq i_d \leq \ell$ such that $\alpha_{i_1} + \cdots + \alpha_{i_d} = \beta$ might not be unique. However, the statement does not depend on the choice. Proposition \ref{prop:basisforRd} is well-known. A proof can be found, for instance, in \cite[Proposition 4.3]{breiding2022algebraic}.

\begin{example}\label{k-basis es} Let $R$ and $\mathcal{F}$ be the algebra and its Khovanskii basis defined in Example \ref{KB example}. The matrix $A$ of leading exponents of $\mathcal{F}$ is
\[ A=\begin{pmatrix} 1 & 1 & 1 & 1 & 1 & 1 \\ 1 & 0 & 1 & 2 & 1 &  2\\ 0 & 2 & 1 & 0 & 2 & 1 \end{pmatrix}. \]
Its columns define a $K-$basis of $R_1$, given by ${\cal F}$. There are 16 distinct 2-element sums of the columns of $A$, and 31 distinct 3-element sums. The corresponding monomials form a basis of the monomial algebra generated by ${\rm in}_\prec({\cal F})$, in degree 2 and 3 respectively. Figure \ref{Picture} shows these monomials in the lattice $\mathbb{Z}^3$. A basis for $R_2$ is
\begin{align*} 
t_0^2 \cdot \{ \phi_1^2,  \phi_0\phi_1,  \phi_1\phi_2,  \phi_1\phi_4,  \phi_0^2,  \phi_0\phi_2,  \phi_0\phi_4,  \phi_1\phi_5,  \phi_4^2,  \phi_0\phi_3,  \phi_0\phi_5,  \phi_2\phi_5,  \phi_4\phi_5,  \phi_3^2,  \phi_3\phi_5,  \phi_5^2 \}.
\end{align*}
This was obtained via Proposition \ref{prop:basisforRd}, by representing each of the dots at level $2$ as a sum of two dots at level $1$. For instance, $(2,0,4) = (1,0,2) + (1,0,2)$ associates the basis element $b_{(2,0,4)} = t_0^2\phi_1^2 \in R_2$ to the monomial $t_0^2t_2^4$.

\end{example}

\begin{figure}[h] 
\centering
\includegraphics[height = 5cm]{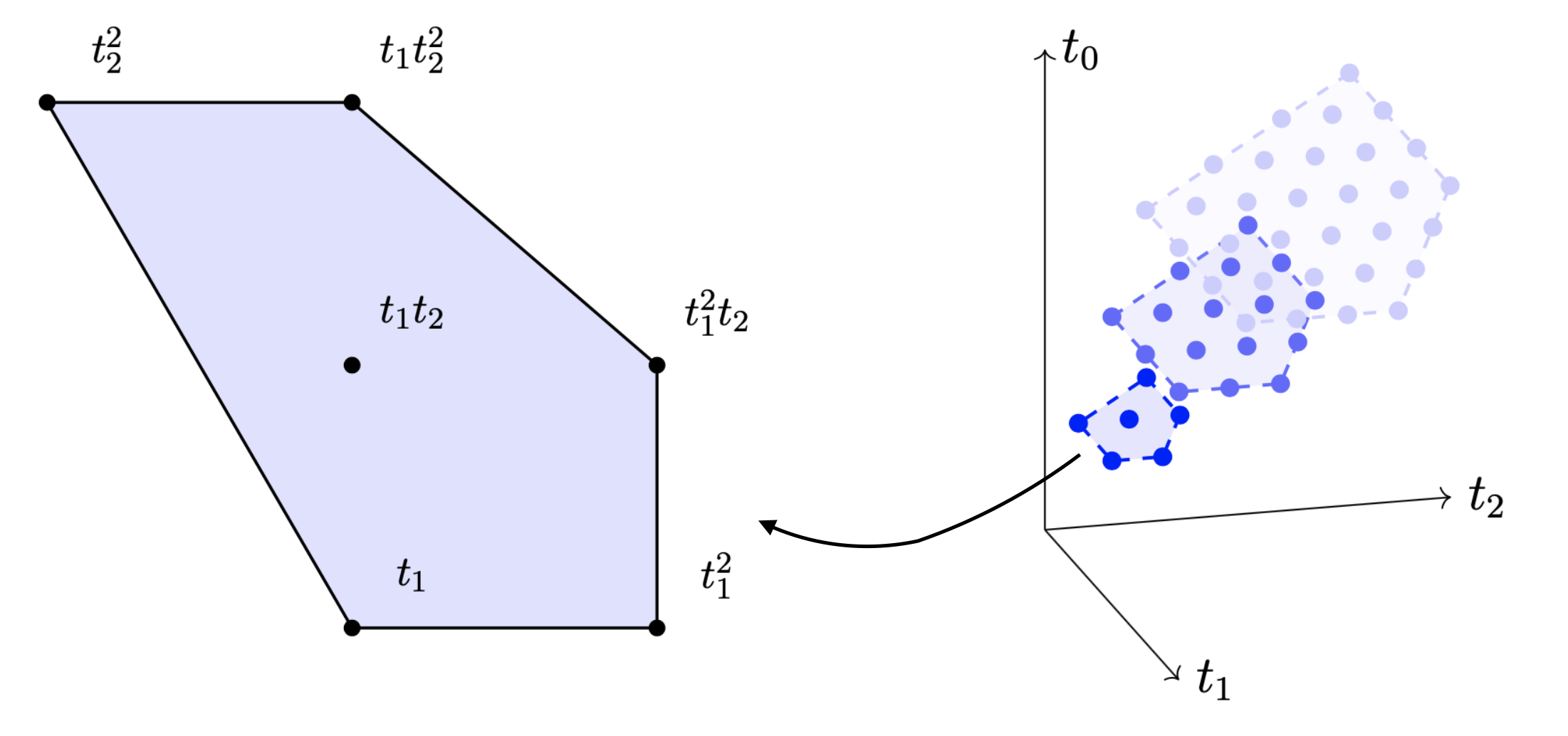}
\caption{Basis of the algebra $R$ from Example \ref{KB example} in degrees $1,2,3$. }
\label{Picture}
\end{figure}

\begin{remark} \label{rem:ehrhart}
In Example \ref{k-basis es}, every lattice point in the convex hull of $2\cdot A$ and $3\cdot A$ corresponds to an element in the basis of $R_2$ and $R_3$, respectively. In general, the inclusion $d \cdot A \subseteq {\rm conv}(d \cdot A) \cap \mathbb{Z}^{n+1}$ might be strict (even for $d =1$). When equality holds for all $d \geq 0$, the Hilbert function ${\rm HF}_R(d)$ is equal to the \emph{Ehrhart polynomial} of $ {\rm conv}(A)$. This happens precisely when the semigroup generated by ${\rm in}_\prec({\cal F})$ is saturated. 
\end{remark}
\begin{example} \label{ex:duffingkhov}
   The set of polynomials $\mathcal{F}= \{t_0\phi_0,\dots, t_0\phi_4\}$ from Example \ref{ex:duffing} is a Khovanskii basis with respect the weight vector $\omega=(1,0,-1)$ for the algebra they generate. The analog of Figure \ref{Picture} for this algebra is \cite[Figure 3]{breiding2022algebraic}. The matrix $A$ is
    \[ A= \begin{pmatrix}
     1 &1 &1  & 1 & 1   \\
        0 &1 &0  & 1 & 0   \\
        0 & 0 & 1 & 2 & 3
    \end{pmatrix}. \]
   The semigroup generated by ${\rm in}_\prec({\cal F})$ is not staturated and $A \subsetneq {\rm conv}(A) \cap \mathbb{Z}^3$. 
\end{example}
\begin{example} \label{ex:grasskhov}
Let ${\rm Gr}(k,m)$ be the Grassmannian of $(k-1)$-planes in $\proj{m-1}_K$ in its Pl\"ucker embedding. Our parameters are $n = k(m-k)$ and $\ell = \binom{m}{k}-1$. The homogeneous coordinate ring $R =K[{\rm Gr}(k,m)]$ is generated by the $k\times k$-minors of a $(k\times m)$-matrix of indeterminates $(t_{ij})$. These form a Khovanskii basis for $R$ with respect to any diagonal term order, i.e., a term order that selects the main diagonal term of the minor as initial term \cite[Theorem 11.8]{sturmfelsgroebner}. 
\end{example}

We close the section by pointing out that if $t_0 \phi_0, \ldots, t_0 \phi_\ell$ form a Khovanskii basis, then the degree of the variety $X$ equals the normalized volume of the polytope ${\rm conv}(A)$. This is implied by \emph{Kushnirenko's theorem} in toric geometry \cite[Theorem 3.16]{telen2022introduction}. In \cite{kaveh2012newton}, Kaveh and Khovanskii generalized this observation. Their result replaces convex lattice polytopes by \emph{Newton-Okounkov bodies}. We do not elaborate on this here.

\section{Khovanskii-Macaulay matrices} \label{sec:3}

We work under the assumption that ${\cal F} = \{ t_0\phi_0, \ldots, t_0 \phi_\ell \}$ is a Khovanskii basis for $K[{\cal F}]$ with respect to some weight vector $\omega \in \mathbb{R}^{n+1}$. As above, $X \subset \mathbb{P}^\ell_{\overline{K}}$ is  the unirational variety parameterized by $\phi$. Recall from the Introduction that our aim is to find all points $x \in X$ satisfying \eqref{eq:homsystem}. In this section, we build structured matrices which are used to compute these points. This will be detailed in Section \ref{sec:4}, where we build an eigenvalue problem from their right nullspace, such that the eigenvalues encode the roots. In the case where $\phi_i$ are monomials, the matrices of this section are the \emph{Macaulay matrices} which are ubiquitous in computer algebra methods for solving polynomial systems \cite{emiris1999matrices,telen2020thesis}. We use the name \emph{Khovanskii-Macaulay matrices}, emphasizing that we exploit the Khovanskii basis structure. We explain the idea by means of an example.

\begin{example} \label{ex:duffing2}
    We turn back to Example \ref{ex:duffing}, and consider the instance 
    \begin{equation}  \label{eq:f1f2} f_1 \,  = \,  1 + 3t_1 + 5t_2 + 7 t_1(t_1^2+t_2^2), \quad f_2 \, = \, 11 + 13t_1 + 17t_2 + 19t_2(t_1^2+t_2^2). 
    \end{equation}
    for concreteness. We present three different Macaulay matrix constructions. The first two are classical. The third one uses Khovanskii bases, and will be used to solve \eqref{eq:homsystem} below. The most standard Macaulay matrix of degree $d$ associated to \eqref{eq:f1f2} has columns indexed by monomials in $t_1, t_2$ of degree at most $d$, and rows by monomial multiples of $f_1, f_2$ of degree at most $d$. Here is an example for $d = 4$: 
    \begin{equation} \label{eq:MP2}
    \kbordermatrix{
    & 1 & t_1 & t_1^2  & t_1^3  & t_1^4 &  t_2  & t_1t_2 & t_1^2t_2 &  t_1^3t_2 &   t_2^2 &  t_1t_2^2 &   t_1^2t_2^2 &   t_2^3 &   t_1t_2^3 &   t_2^4\\
    f_1 &1 & 3 & 0 & 7 & 0 & 5 & 0 & 0 & 0 & 0 & 7 & 0 & 0 & 0 & 0 \\
t_1 \cdot f_1&0 & 1 & 3 & 0 & 7 & 0 & 5 & 0 & 0 & 0 & 0 & 7 & 0 & 0 & 0 \\
t_2 \cdot f_1 &0 & 0 & 0 & 0 & 0 & 1 & 3 & 0 & 7 & 5 & 0 & 0 & 0 & 7 & 0 \\
f_2&11 & 13 & 0 & 0 & 0 & 17 & 0 & 19 & 0 & 0 & 0 & 0 & 19 & 0 & 0 \\
t_1 \cdot f_2&0 & 11 & 13 & 0 & 0 & 0 & 17 & 0 & 19 & 0 & 0 & 0 & 0 & 19 & 0 \\
t_2 \cdot f_2 &0 & 0 & 0 & 0 & 0 & 11 & 13 & 0 & 0 & 17 & 0 & 19 & 0 & 0 & 19 \\
    }
    \end{equation}
    The same matrix is obtained by homogenizing $f_1$ and $f_2$, and indexing the columns by the 15 monomials in three variables of degree 4 (as opposed to \emph{at most} 4). Matrices like \eqref{eq:MP2} are used to solve $f_1 = f_2 = 0$ on $\mathbb{P}^2$, which is why we denote it by $M_{\mathbb{P}^2}(4)$. 
    
    Now consider the equations \eqref{eq:homsysduffing}. The surface $X$ is defined by 3 polynomials: 
    \begin{equation} \label{eq:implicitduffing} x_1x_4 - x_2x_3 \, = \, x_1^2x_2 + x_2^3 - x_4x_0^2 \, = \, x_1^3 + x_1x_2^2 - x_3x_0^2 \, = \, 0 \quad \text{in } \mathbb{P}^{4}.  \end{equation}
    In accordance with \eqref{eq:f1f2}, the polynomials $F_i$ are defined as follows: 
    \begin{equation} \label{eq:F1F2}
    F_1 \, = \, x_0 + 3 \, x_1 + 5 \, x_2 + 7 \, x_3, \quad F_2 \, = \, 11 x_0 + 13 \, x_1 + 17 \, x_2 + 19 \, x_4. 
    \end{equation}
    The ideal $J$ of the solutions to \eqref{eq:homsystem} is generated by $F_1, F_2$, together with the three polynomials in \eqref{eq:implicitduffing}. The Macaulay matrix in degree $2$ for this set of generators is 
    \small
\[
  \kbordermatrix{
& x_0^2 & x_0x_1 & x_1^2 & x_0x_2 & x_1x_2 & x_2^2 & x_0x_3 & x_1x_3 & x_2x_3 & x_3^2 & x_0x_4 & x_1x_4 & x_2x_4 & x_3x_4 & x_4^2 \\
x_0 \cdot F_1 &1 & 3 & 0 & 5 & 0 & 0 & 7 & 0 & 0 & 0 & 0 & 0 & 0 & 0 & 0 \\
x_1 \cdot F_1  &0 & 1 & 3 & 0 & 5 & 0 & 0 & 7 & 0 & 0 & 0 & 0 & 0 & 0 & 0 \\
x_2 \cdot F_1 &0 & 0 & 0 & 1 & 3 & 5 & 0 & 0 & 7 & 0 & 0 & 0 & 0 & 0 & 0 \\
\vdots & \vdots & \vdots & \vdots & \vdots & \vdots & \vdots & \vdots & \vdots & \vdots & \vdots & \vdots & \vdots & \vdots & \vdots & \vdots\\
x_3 \cdot F_2 &0 & 0 & 0 & 0 & 0 & 0 & 11 & 13 & 17 & 0 & 0 & 0 & 0 & 19 &  0 \\
x_4 \cdot F_2 &0 & 0 & 0 & 0 & 0 & 0 & 0 & 0 & 0 & 0 & 11 & 13 & 17 & 0 & 19 \\
x_1x_4-x_2x_3 &0 & 0 & 0 & 0 & 0 & 0 & 0 & 0 & -1 & 0 & 0 & 1 & 0 & 0 & 0 \\
  }
\]
\normalsize

\noindent with rows indexed by degree 2 elements $g$ of the ideal $J$, and columns indexed by monomials $x^\alpha$ of degree 2. Denoting this matrix by $M_{\mathbb{P}^4}(2)$ with entries $M_{\mathbb{P}^4}(2)_{g,x^\alpha}$, we read off that $g = \sum_{\alpha} M_{\mathbb{P}^4}(2)_{g,x^\alpha} \, x^\alpha$. This construction represents graded pieces of $J$ in the coordinate ring $S = K[x_0, x_1, x_2,x_3,x_4]$ of $\mathbb{P}^4$. Clearly, the matrices $M_{\mathbb{P}^4}(d)$ grow much faster than $M_{\mathbb{P}^2}(d)$ with $d$. There is an important geometric difference between these two constructions. For large enough $d$, the matrix $M_{\mathbb{P}^2}(d)$ can be used to compute the 9 solutions of $f_1 = f_2 = 0$ on $\mathbb{P}^2$. Here \emph{large enough} means $d \geq 5$, see below. On the other hand, for large enough $d$, $M_{\mathbb{P}^4}(d)$ can be used to compute the 5 solutions of $f_1 = f_2 = 0$ on $X \subset \mathbb{P}^4$. This time \emph{large enough} means $d \geq 3$. Computing the 5 solutions on $X$ is more true to our goal in \eqref{eq:homsystem}, so we want to keep this feature.

What we propose in this paper is an alternative construction $M_X(d)$ which works directly in the coordinate ring $K[X] = S/I(X)$ of $X$. This reduces the size of the matrix, while still computing on $X$. Our ideal is $\langle F_1, F_2 \rangle \subset K[X]$. In degree 2 we get
\small
\[
 \kbordermatrix{%
& x_0^2 & x_0x_1 & x_1^2 & x_0x_2 & x_1x_2 & x_2^2 & x_0x_3 & x_1x_3 & x_2x_3 & x_3^2 & x_0x_4 & x_2x_4 & x_3x_4 & x_4^2 \\
x_0 \cdot F_1 &1 & 3 & 0 & 5 & 0 & 0 & 7 & 0 & 0 & 0 & 0  & 0 & 0 & 0 \\
x_1 \cdot F_1 &0 & 1 & 3 & 0 & 5 & 0 & 0 & 7 & 0 & 0 & 0  & 0 & 0 & 0 \\
x_2 \cdot F_1 &0 & 0 & 0 & 1 & 3 & 5 & 0 & 0 & 7 & 0 & 0  & 0 & 0 & 0 \\
x_3 \cdot F_1 &0 & 0 & 0 & 0 & 0 & 0 & 1 & 3 & 5 & 7 & 0  & 0 & 0 & 0 \\
x_4 \cdot F_1 &0 & 0 & 0 & 0 & 0 & 0 & 0 & 0 & 3 & 0 & 1  & 5 & 7 & 0 \\
x_0 \cdot F_2 &11& 13 & 0 & 17 & 0 & 0 & 0 &0 &0 & 0 & 19  & 0 & 0 & 0 \\
x_1 \cdot F_2 &0 & 11 & 13 & 0 & 17 & 0 & 0 &0& 19 & 0 & 0  & 0 & 0 & 0 \\
x_2 \cdot F_2 &0 & 0 & 0 & 11 & 13 & 17 & 0&0 & 0 & 0 & 0  & 19 & 0 & 0 \\
x_3 \cdot F_2 &0 & 0 & 0 & 0 & 0 & 0 & 11&13 & 17 & 0 & 0  & 0 & 19 & 0 \\
x_4 \cdot F_2 &0 & 0 & 0 & 0 & 0 & 0 & 0 & 0 & 13 & 0 & 11  & 17 & 0 & 19 \\
  }.
\]
\normalsize

\noindent This matrix, denoted $M_X(2)$, is slightly smaller than $M_{\mathbb{P}^4}(2)$. Its columns are indexed by a basis for $K[X]_2$, which has 14 elements instead of 15. There is no column indexed by $x_1x_4$. 
The fact that $x_1x_4 - x_2x_3 = 0$ on $X$ has been taken into account by adding the column previously indexed by $x_1x_4$ to column $x_2x_3$. 
For general $d$, the rows of $M_{X}(d)$ will be indexed by all monomial multiples $x^\alpha \cdot F_i$, where $x^\alpha$ runs over a basis of $K[X]_{d-{\rm deg}(F_i)}$.
The matrices $M_X(d)$ can be constructed for any projective variety $X \subset \mathbb{P}^\ell$. However, for it to be practical, we need an efficient way of constructing a basis for $K[X]_d$.
We did this for $K[X]_2$ via Proposition \ref{prop:basisforRd}, using the fact that ${\cal F} = \{t_0 \phi_0, \ldots, t_0  \phi_4 \}$ is a Khovanskii basis for $K[{\cal F}]$ (Example \ref{ex:duffingkhov}). The basis consists~of 
\begin{equation}  \label{eq:xtot} 
x_0^{\alpha_0}x_1^{\alpha_1}x_2^{\alpha_2}x_3^{\alpha_3}x_4^{\alpha_4} \, \longleftrightarrow \, b_\beta \, = \, t_0^2 \cdot \phi_0(t)^{\alpha_0}\phi_1(t)^{\alpha_1}\phi_2(t)^{\alpha_2}\phi_3(t)^{\alpha_3}\phi_4(t)^{\alpha_4}, 
\end{equation}
where $x^\alpha$ runs over the 14 monomials indexing the columns of $M_X(2)$. 
The $\omega$-leading term of this basis element is $t^\beta$, where $\beta = \sum_{i=0}^4 \alpha_i \cdot {\rm in}_\omega(t_0 \cdot \phi_i) \in 2 \cdot A$. The last row of $M_X(2)$ reads $x_4 \cdot F_2 = 13 x_2x_3 + 11 x_0x_4 + 17 x_2 x_4 + 19 x_4^2 \mod I(X)$. This identity can be verified by plugging in $x_i = \phi_i(t)$. To emphasize the role of Khovanskii bases in the construction of our matrices, we call $M_X(d)$ a \emph{Khovanskii-Macaulay matrix}. 

We have now seen three matrix constructions $M_{\mathbb{P}^2}(d), M_{\mathbb{P}^4}(d), M_{X}(d)$ associated to the equations $f_1 = f_2 = 0$. To compare their efficiency, we need to know for which degree $d$ these matrices allow us to compute the solutions. We will answer this more generally in the case where $f_1$ and $f_2$ are generic and homogeneous of degree $d$ in the $\phi_i$, cf.~\eqref{eq:system}. In the case of $\mathbb{P}^2$, one needs to use the \emph{Macaulay bound} $\deg(f_1) + \deg(f_2) - 1$, where $\deg(f_i) = 3d$ is the degree in the $t$-variables \cite[Algorithm 4.2]{telen2020thesis}. Hence, we must construct $M_{\mathbb{P}^2}(6d - 1)$. The number of solutions in $\mathbb{P}^2$ is $9d^2$. Out of these, $4d^2$ lie ``at infinity''. For the other two constructions, we will show in Section \ref{sec:5} that it suffices to use $M_{\mathbb{P}^4}(2d+1)$, and $M_X(2d+1)$. The number of solutions on $X$ is $5d^2$. Here is a summary of the three matrix sizes for small $d$:

\begin{center}
\begin{tabular}{c|cccc}
$d$                   & 1              & 2                & 3                & 4                \\ \hline
$M_{\mathbb{P}^2}(6d-1)$ & $12 \times 21$ & $42 \times 78$   & $90 \times 171$  & $156 \times 300$ \\
$M_{\mathbb{P}^4}(2d+1)$ & $37 \times 35$ & $135 \times 126$ & $406 \times 330$ & $1002\times 715$ \\
$M_{X}(2d+1)$            & $28 \times 28$ & $56 \times 71$   & $94 \times 134$  & $142 \times 217$
\end{tabular}
\end{center}

Notice that $M_X(2d+1)$ quickly becomes the smallest among these three matrices when $d$ increases. The size of $M_\bullet(d)$ grows like the Hilbert function of $\bullet$. 
\end{example}

With this example in mind, we now define Khovanskii-Macaulay matrices for our general setup from the Introduction. Let $X \subset \mathbb{P}^\ell$ be the unirational variety parameterized by $\phi$. Suppose ${\cal F} = \{t_0 \phi_0, \ldots, t_0 \phi_\ell \}$ is a Khovanskii basis for $K[{\cal F}] \subset K[t_0, \ldots, t_n]$. Let $A \in (n+1) \times (\ell + 1)$ be the matrix of leading exponents of ${\cal F}$. For any positive integer $d$, let $\{b_{d,\beta}\}_{\beta \in d \cdot A}$ be the basis of $K[X]_d = K[{\cal F}]_d$ from Proposition \ref{prop:basisforRd}.

\begin{defn}
Let $f_i, i = 1, \ldots, s$ be as in \eqref{eq:fi}. The \emph{Khovanskii-Macaulay matrix} (KM matrix) of $(f_1, \ldots, f_s)$ in degree $d$ has rows indexed by $(i,\gamma)$, with $i \in \{1, \ldots, s\}, \gamma \in (d-d_i) \cdot A$, and columns indexed by $\beta \in d \cdot A$. Its entries $M_X(d)_{(i,\gamma),\beta}$ are defined by  
    \begin{equation} \label{eq:expandmac}
    b_{d-d_i,\gamma} \cdot t_0^{d_i} \cdot  f_i \, = \, \sum_{\beta \in d \cdot A} M_X(d)_{(i,\gamma),\beta} \,  b_{d,\beta}. \end{equation}
\end{defn}
The rows of $M_X(d)$ simply expand $b_{d-d_i,\gamma} \cdot  t_0^{d_i} \cdot  f_i \in K[{\cal F}]_d$ in the basis $\{b_{d,\beta}\}_{\beta \in d \cdot A}$. The number of rows of $M_X(d)$ is $\sum_{i = 1}^s |(d-d_i) \cdot A| = \sum_{i=1}^s {\rm HF}_X(d-d_i)$, and the number of columns is $|d \cdot A| = {\rm HF}_X(d)$.
Here are some easy facts about KM matrices. 
\begin{prop} \label{prop:macprop}
    The row span of $M_X(d)$ is isomorphic to $I_d$, where $I = \langle F_1, \ldots, F_s \rangle \subset K[X]$ is generated by the $F_i$ from \eqref{eq:homsystem}. The rank of $M_X(d)$ equals ${\rm HF}_{I}(d)$ and the (right) kernel of $M_X(d)$ has dimension ${\rm HF}_{K[X]/I}(d)$. 
\end{prop}

Our next section discusses how to use $M_X(d)$ for solving the equations \eqref{eq:homsystem}.

\section{Eigenvalue theorem} \label{sec:4}
This section shows how to pass from a Khovanskii-Macaulay matrix $M_X(d)$, for large enough $d$, to an eigenvalue problem which reveals the solutions of \eqref{eq:homsystem}. In order to construct $M_X(d)$ as in Section \ref{sec:3}, we keep assuming that ${\cal F} = \{ t_0\phi_0, \ldots, t_0 \phi_\ell \}$ is a Khovanskii basis for $K[{\cal F}]$. The polynomials $f_i, F_i$ and the variety $X$ are those from the Introduction. We also assume that $I = \langle F_1, \ldots, F_s \rangle \subset K[X]$ defines a zero-dimensional subscheme $V_X(I)$ of $X$ of degree $\delta < \infty$. 
The following definition helps to clarify \emph{large enough} in this section's first sentence. It uses \emph{saturation} of $I$, which is 
\begin{equation}\label{def:saturation}
I^{\rm sat} \, = \, \{ f \in K[X] \, : \, K[X]_d \cdot f \subset I, \text{ for some } d \in \mathbb{N} \}.     
\end{equation} 

\begin{defn} \label{def:regularity}
Let $X, I, I^{\rm sat}$ be as above. The \emph{regularity} of $I$ is the set of degrees
\[ {\rm Reg}(I) \, = \, \{ d \in \mathbb{Z} \, : \, {\rm HF}_{K[X]/I}(d) \, = \,  \delta \text{ and } I_d \, = \, I_d^{\rm sat} \}. \]
\end{defn}
\begin{example} \label{ex:regduffing}
    Let $I = \langle F_1, F_2 \rangle \subset K[X]$ be the ideal generated by $F_1,F_2$ from \eqref{eq:F1F2} in the coordinate ring of our surface from Example \ref{ex:duffing2}. We verify using \texttt{Macaulay2} that $I = I^{\rm sat}$ and the Hilbert function of $K[X]/I$ in degrees $0,1,2,3,4,\ldots$ equals $1,3,5,5,5,\ldots$. We conclude that ${\rm Reg}(I) = \{ d \in \mathbb{Z} \, : \, d \geq 2\}$.
\end{example}
We will see below (Theorem \ref{thm:eigenvaluethm2}) that we will end up working with $M_{X}(d+e)$, where $e >0$ and both $d$ and $d+e$ are contained in the regularity ${\rm Reg}(I)$. 

From Definition \ref{def:regularity}, it follows in particular that when $d+e \in {\rm Reg}(I)$, the kernel of $M_X(d+e)$ has dimension $\delta$ (Proposition \ref{prop:macprop}). This kernel plays an important role for solving. The reason will become clear in Theorem \ref{thm:eigenvaluethm2}. For any degree $t$, we write $N_X(t)$ for the transpose of a kernel matrix of $M_X(t)$. That is, $N_X(t)$ is of size ${\rm HF}_{K[X]/I}(t) \times {\rm HF}_{K[X]}(t)$, and it has rank ${\rm HF}_{K[X]/I}(t)$. For $t = d+e \in {\rm Reg}(I)$, this matrix represents a map $N_X(d+e): K[X]_{d+e} \longrightarrow K^\delta$. Let $N = N_X(d+e)$ and let $p \in K[X]_{e}$ be any element of degree $e$. We define a linear map 
\begin{equation} \label{eq:Nh} N_p : K[X]_{d} \longrightarrow K^{\delta}, \quad \text{by setting} \quad N_p(g) = N(gp). \end{equation}

\begin{example}[$d = 2, e = 1$] \label{ex:duffing3}
We continue Example \ref{ex:duffing2} and describe a particular choice for the kernel matrix $N_X(3)$ of $M_X(3)$. We have seen in Example \ref{ex:regduffing} that ${\rm HF}_{K[X]/I}(3) = 5$. Hence, the matrix $N=N_X(3)$ has size $5\times 28$ and rank $5$. Each row of $N$ represents a linear functional vanishing on $I_3$. Some natural candidates for such functionals are the evaluations $g \mapsto g(z_i)$ at the five solutions $\{z_1, \ldots, z_5\} = V_X(I)$. In fact, we will see that these form a basis of the kernel. Hence, we can take $N$ to be 
\[ N:K[X]_3\rightarrow K^5, \quad N(g)=\big (
    g(z_1),  \dots , g(z_5)\big)^T. \]
This is represented by the matrix $(z_1^{{\cal M}_3}~ \ldots~ z_5^{{\cal M}_3})^T$, where the column vectors $z_i^{{\cal M}_3} \in \mathbb{C}^{28}$ are constructed as follows. The columns of $M_X(3)$ are indexed by a set of monomials ${\cal M}_3$ which is a basis of $K[X]_3$. This is the third matrix in Example \ref{ex:duffing2}. We let $z_i^{{\cal M}_3}$ be the vector of these monomials evaluated at the $i$-th solution $z_i$. More precisely, we represent the solution $z_i$ by any set of projective coordinates in $\mathbb{P}^4$, and plug this into our monomials. Picking a different set of projective coordinates would only scale the $i$-th row of $N$, and thus lead to an alternative kernel map. 

For a polynomial $p\in K[X]_1$ we consider the linear map $N_p$ given by 
\[ N_p:K[X]_2\rightarrow K^5, \quad N_p(g)=N(gp)=\big(
    g(z_1)p(z_1),\dots , g(z_5)p(z_5)\big )^T. \]
Letting ${\cal M}_2$ be a monomial basis of $K[X]_2$, we observe that $N_p$ is represented by 
\[N_p={\rm diag}\begin{pmatrix} p(z_1) ~ \cdots ~ p(z_5) \end{pmatrix}\cdot\begin{pmatrix} z_1^{{\cal M}_2} & \dots & z_5^{{\cal M}_2}
\end{pmatrix}^T.  \]
Here the first factor is a $5 \times 5$ diagonal matrix with diagonal entries $p(z_1), \ldots, p(z_5)$.
\end{example}
 Example \ref{ex:duffing3} hints at a general way to represent the kernel $N_X(d)$, for $d \in {\rm Reg}(I)$ in the case where $V_X(I)$ defines a set of $\delta$ reduced points on $X$.
\begin{lemm} \label{lem:basis}
    Suppose $I \subset K[X]$ defines a reduced subscheme $V_X(I)$ of $X$, which consists of $\delta < \infty$ points $\{z_1, \ldots, z_\delta\} \subset X$. Let $d \in {\rm Reg}(I)$, and let ${\cal M}_d = \{ x^{\alpha_1}, \ldots, x^{\alpha_{m}} \}$ be the basis for $K[X]_d$ indexing the columns of $M_X(d)$. The vectors $z_i^{{\cal M}_d}  = (z_i^\alpha)_{\alpha \in {\cal M}(d)}$ form a basis for $N_X(d) = \ker M_X(d)$.   
\end{lemm}
\begin{proof}
We first show that a homogeneous element $g \in K[X]$ is contained in $I^{\rm sat}$ \eqref{def:saturation} if and only if $g(z_i) = 0$ for $i = 1, \ldots, \delta$. 

    If $g\in I^{\rm sat}$ is homogeneous, then $K[X]_e\cdot g\subset I$ for some $e\in \N$ and, in particular, $x_i^e\cdot g\in I$ for all $i=0,\dots,\ell$. For every $z \in V_X(I)$ we can find a suitable $0 \leq j \leq \ell$ such that the $j$-th homogeneous coordinate $z_j$ is nonzero. We may thus assume $z_j=1$ and we have $0=(x_j^e\cdot g )(z)=g(z)$. On the other hand, suppose that $I=\langle F_1,\dots,F_s \rangle$ and $g$ is a homogeneous polynomial satisfying $g(z)=0$ for each $z\in V_X(I)$. 
We consider the dehomogenization $\hat{g}_i= g(\frac{x_0}{x_i},\dots, \frac{x_{i-1}}{x_i}, 1, \frac{x_{i+1}}{x_i},\dots,\frac{x_\ell}{x_i})$. Since $V_X(I)$ is reduced, $\hat{g}_i$ belongs to the dehomogenized ideal $\langle \hat{F}_{1_i},\dots,\hat{F}_{s_i}\rangle\subset (K[X]_{x_i})_0$ for each $i=0,\dots,\ell $:
\begin{equation} \label{dehomog}
    \hat{g}_i= \hat{h}_1 \hat{F}_{1_i}+\dots + \hat{h}_s \hat{F}_{s_i},
\end{equation}  
Pick $\alpha\in\N$ large enough such that $x_i^{\alpha}$ clears denominators in every term of \eqref{dehomog} for every $i$. In particular we have that $\alpha> {\rm deg}(g) $ and 
\[ x_i^\alpha \hat{g_i}= x_i^{\alpha-\deg(g)}g \in I. \]
It follows that, for $e=(\ell+1)(\alpha-\deg(g)-1)+1$, $K[X]_e\cdot g \in I$ and $g$ belongs to $I^{\rm sat}$.

   By Proposition \ref{prop:macprop}, the kernel of $M_X(d)$ consists of functionals $v \in {\rm Hom}_K(K[X]_d,K)$ such that $v(I_d) = 0$. Since $d \in {\rm Reg}(I)$, we have that $I_d = I^{\rm sat}_d$, and hence $g \in I_d$ if and only if the evaluation functional ${\rm ev}_{z_i}: g \mapsto g(z_i)$ vanishes on $g$. Hence $\ker M_X(d)$ is generated by ${\rm ev}_{z_1}, \ldots, {\rm ev}_{z_\delta}$. These functionals are represented by the vectors $z_i^{{\cal M}_d}$ from the lemma. Since $d \in {\rm Reg}(I)$, the dimension of $\ker M_X(d)$ equals $\delta$. Therefore, the vectors $z_i^{{\cal M}_d}$ are linearly independent, and they form a basis for $\ker M_X(d)$. 
\end{proof}

Our main theorem in this section states that, if $d, d+e \in {\rm Reg}(I)$, then from the kernel $N_X(d+e)$ of the Khovanskii-Macaulay matrix $M_X(d+e)$ we can construct eigenvalue problems which reveal the solutions.
Suppose $I \subset K[X]$ defines a reduced subscheme $V_X(I)$ of $X$, which consists of $\delta < \infty$ points. Let $d, d+e \in {\rm Reg}(I)$ and consider $h \in K[X]_e$ such that $h$ is non-vanishing at the points in $V_X(I)$. Let $N: K[X]_{d+e} \rightarrow K^\delta$ be given by $N_X(d+e)$, and $N_p: K[X]_d \rightarrow K^\delta$ is defined as in \eqref{eq:Nh}.

\begin{thm} \label{thm:eigenvaluethm2}
 With the above assumptions and notation, there exists a subspace $B \subset K[X]_d$ such that the restriction $(N_h)_{|B}$ is invertible. Moreover, for any $p \in K[X]_e$, the eigenvalues of $(N_h)_{|B}^{-1} \circ (N_p)_{|B} : B \longrightarrow B$ are $\frac{p}{h}(z)$, for $z \in V_X(I)$. 
\end{thm}

\begin{proof}
    Let $V_X(I) = \{ z_1, \ldots, z_\delta \}$. The matrix $M_X(d+e)$ has columns indexed by a set of monomials ${\cal M}_{d+e}$ which are a basis for $K[X]_{d+e}$. We denote by $z_i^{{\cal M}_{d+e}}$ the vector whose entries are these monomials evaluated at $z_i \in V_X(I)$. The vectors $\{z_1^{{\cal M}_{d+e}}, \ldots, z_\delta^{{\cal M}_{d+e}} \}$ form a basis for the right kernel of $M_X(d+e)$ (Lemma \ref{lem:basis}).
    This means we can use these vectors for the rows of $N=N_X(d+e)$, such that it represents the map $K[X]_{d+e} \rightarrow K^\delta$ given by $N(g) = (g(z_1), \ldots, g(z_\delta))$. 
   Like in Example \ref{ex:duffing3}, if ${\cal M}_{d}$ is a monomial basis for $K[X]_d$, the map $N_p: K[X]_d \rightarrow K^\delta$ is represented by
   \[ N_p={\rm diag}\begin{pmatrix} p(z_1) ~ \cdots ~ p(z_\delta) \end{pmatrix}\cdot\begin{pmatrix} z_1^{{\cal M}_{d}} & \dots & z_\delta^{{\cal M}_{d}}
\end{pmatrix}^T.\]
    By Lemma \ref{lem:basis}, the matrix $\begin{pmatrix} z_1^{{\cal M}_{d}} & \dots & z_\delta^{{\cal M}_{d}}
\end{pmatrix}^T$ has rank $\delta$. If $h \in K[X]_e$ is such that $h(z_i) \neq 0$ for all $z_i \in V_X(I)$, we can find a subset of monomials $\{ x^{\beta_1}, \ldots, x^{\beta_\delta} \} \subset {\cal M}_d$ such that the submatrix of $N_h$ whose columns are indexed by $x^{\beta_1}, \ldots, x^{\beta_\delta}$ is invertible. That submatrix is $(N_h)_{|B}$, where $B$ is the $K$-span of $x^{\beta_1}, \ldots, x^{\beta_\delta}$. Let us write $Z$ for the corresponding invertible $\delta \times \delta$ submatrix of $\begin{pmatrix} z_1^{{\cal M}_{d}} & \dots & z_\delta^{{\cal M}_{d}}
\end{pmatrix}^T$, such that $(N_p)_{|B} = {\rm diag}\begin{pmatrix} p(z_1) ~ \cdots ~ p(z_\delta) \end{pmatrix} \cdot Z$. The composition $(N_h)_{|B}^{-1} \circ (N_f)_{|B}$ is 
    \[ Z\inv \cdot {\rm diag} \begin{pmatrix} \frac{f(z_1)}{h(z_1)}  & \cdots & \frac{f(z_\delta)}{h(z_\delta)} \end{pmatrix}  \cdot Z, \]
    which proves the claim about the eigenvalues.
\end{proof}

Theorem \ref{thm:eigenvaluethm2} is at the heart of our eigenvalue method for solving $F_1 = \cdots = F_s = 0$ on $X$: we use the eigenstructure of the matrices $(N_h)_{|B}^{-1} (N_p)_{|B}$ to find the coordinates of the solutions $z_i$. It is a generalization of the classical \emph{eigenvalue-eigenvector theorem} from computational algebraic geometry \cite[Chapter 2, \S 4, Theorem 4.5]{cox2006using}. Like that theorem, it extends to the case with multiplicities. That is, one can drop the reducedness assumption. We omit this here to keep the presentation simple. 
The proof of Theorem \ref{thm:eigenvaluethm2} shows that the matrices $(N_h)_{|B}^{-1} (N_p)_{|B}$ share a set of eigenvectors $Z$, and they commute pair-wise. We discuss how to use Theorem \ref{thm:eigenvaluethm2} in more detail in Section \ref{sec:6}. The next section is concerned with determining ${\rm Reg}(I)$ in the case of complete intersections.

\section{Regularity} \label{sec:5}
We investigate the regularity ${\rm Reg}(I)$ from Definition \ref{def:regularity} in the case where $I \subset K[X]$ is generated by $n = \dim X$ elements. While we work on unirational varieties with Khovanskii basis parameterizations in other sections, we here only need the following. 
\begin{assum} \label{assum:ACM}
The projective variety $X \subset \mathbb{P}^\ell$ is arithmetically Cohen-Macaulay. That is, its homogeneous coordinate ring $K[X]$ is a Cohen-Macaulay ring. 
\end{assum} 
Note that if $X$ is parameterized by a Khovanskii basis ${\cal F}$, then $K[X]$ is Cohen-Macaulay if the same holds for the toric algebra $K[{\rm in}_\prec({\cal F})]$, which is easy to check.

Assumption \ref{assum:ACM} makes the saturation condition in Definition \ref{def:regularity} trivial: 
\begin{lemm} \label{lem:IisIsat}
Let $X$ satisfy Assumption \ref{assum:ACM} and let $I = \langle F_1, \ldots, F_n \rangle \subset K[X] $ be a homogeneous ideal such that $V_X(I)$ is zero-dimensional. Then we have $I=I^{\rm sat}$.
\end{lemm}
\begin{proof}
Consider the primary decomposition $I=Q_1 \cap \cdots \cap Q_s$ and suppose that $f\in I^{\rm sat}\setminus I $. We may assume that $f\not\in Q_1$. There exists some $d\in \N$ such that $K[X]_d\cdot f\subseteq Q_1$, hence we have that $K[X]_d\subseteq \sqrt{Q_1}$, and the latter ideal has codimension $n$ by the Unmixedness Theorem \cite[Corollary 18.14]{eisenbud1995commutative}. Then we have an inclusion $V(K[X]_d)\supseteq V(\sqrt{Q_1})$ between varieties of codimension greater than $n$ and equal to $n$, respectively, that is a contradiction. It follows that $I=I^{\rm sat}$.
\end{proof}
A first bound on ${\rm Reg}(I)$ is expressed in terms of the \emph{Hilbert regularity} of $K[X]$.
\begin{defn}\label{def:hilbreg}
Let $M$ be a graded $K[X]$-module with Hilbert polynomial ${\rm HP}_M:\mathbb{Z} \rightarrow \mathbb{Z}$. The \emph{Hilbert regularity} ${\rm HReg}(M)$ of $M$ is the smallest integer $d$ such that ${\rm HF}_M(e) = {\rm HP}_M(e)$ for all $e \geq d$.
\end{defn}
As before, let $I = \langle F_1, \ldots, F_n \rangle$ with ${\rm deg}(F_i) = d_i$. Let $[n] = \{ 1, \ldots, n \}$ and for a subset $J \subset [n]$, we write $d_J = \sum_{i \in J} d_i$. 
\begin{thm} \label{thm:regfromhreg}
    Let $X$ satisfy Assumption \ref{assum:ACM} and let $I = \langle F_1, \ldots, F_n \rangle \subset K[X]$ be a homogeneous ideal with ${\rm deg}(F_i) = d_i$, such that $V_X(I)$ is zero-dimensional. We have the inclusion $\{d_{[n]} + w \in \mathbb{Z} \, : \, w \geq {\rm HReg}(K[X]) \} \subset {\rm Reg}(I)$.
\end{thm}
\begin{proof}
    By Assumption \ref{assum:ACM} and ${\rm dim} \, V_X(I) = 0$, the polynomials $F_1, \ldots, F_n$ form a regular sequence in $K[X]$. Therefore, a free resolution of the ideal $I$ is given by the Koszul complex. Restricted to degree $d_{[n]} + w$, this complex is
    \begin{equation} \label{eq:koszulACM}
    0 \rightarrow K[X]_k \rightarrow \bigoplus_{i=1}^n K[X]_{d_i+w}  \rightarrow  
    \cdots \rightarrow \bigoplus_{J \in \binom{[n]}{n-1}} K[X]_{d_J+w}
    \rightarrow I_{d_{[n]} + w} \rightarrow 0.
    \end{equation}
    This easily implies the following alternating-sum dimension formula for all $w$:
    \[ {\rm dim}_K (K[X]/I)_{d_{[n]}+w} \, = \, \sum_{p= 0}^n (-1)^{n-p} \cdot \sum_{J \in \binom{[n]}{p}} \dim_K K[X]_{d_J + w}. \]
    If $w \geq {\rm HReg}(K[X])$, the right hand side equals the polynomial
    \[ \sum_{p= 0}^n (-1)^{n-p} \cdot \sum_{J \in \binom{[n]}{p}} {\rm HP}_{K[X]}(d_J + w). \]
    This agrees with ${\rm HP}_{K[X]/I}(d_{[n]} + w)$ for large enough $w$, hence it must equal the constant polynomial $\delta$ in Definition \ref{def:regularity}. Since $I = I^{\rm sat}$ (Lemma \ref{lem:IisIsat}), the theorem follows. 
\end{proof}

\begin{example}
Let $X \subset \mathbb{P}^2$ be the curve $\{y^3-xz^2 = 0 \}$. The Hilbert polynomial is ${\rm HP}_{K[X]}(t) = 3 \cdot t$, and the Hilbert regularity is ${\rm HReg}(K[X]) = 1$. The ideal $I = \langle c_1 x + c_2 y + c_3 z \rangle \subset K[X]$ defines $\delta = 3$ points on $X$. We have ${\rm HF}_{K[X]/I}(t) = {\rm HF}_{K[X]}(t) - {\rm HF}_{K[X]}(t-1)$, which for $t = 0, 1, 2, 3,\ldots$ gives $1, 2, 3, 3, \ldots$. Hence, the inclusion $\{ d_{[1]} + k : k \geq 1 \} = \{2,3,4,\ldots \} \subset {\rm Reg}(I)$ from Theorem \ref{thm:regfromhreg} is an equality in this example. 
\end{example}

Theorem \ref{thm:regfromhreg} only leads to effective degree bounds for ${\rm Reg}(I)$ if the Hilbert regularity ${\rm HReg}(K[X])$ can be easily computed.  
\begin{example}[lines in $\mathbb{P}_K^{3}$]
    The Hilbert regularity might be negative. The Grassmannian ${\rm Gr}(2,4)$ of lines in $\mathbb{P}_K^3$ is a hypersurface $X \subset \mathbb{P}^5$ via its Pl\"ucker embedding. Its defining ideal is generated by a single polynomial, namely the Pl\"ucker quadric. One easily computes that ${\rm HP}_{K[X]}(t) = \frac{1}{12}(t+1)(t+2)^2(t+3)$, and ${\rm HReg}(K[X]) = -3$: 
    \begin{center}
\begin{tabular}{c|ccccccccccccc}
 $t$ & $-6$ &  $-5$ & $-4$   & $-3$ & $-2$ & $-1$ & 0 & 1 & 2 & 3 & 4 & 5 & 6  \\ \hline
 ${\rm HF}_{K[X]}(t)$& 0  & 0  & 0  & 0 & 0 & 0 & 1 & 6 & 20 & 50 & 105 & 196 & 336  \\
 ${\rm HP}_{K[X]}(t)$& 20 & 6 & 1 & 0 & 0 & 0 & 1  & 6 & 20 & 50 & 105 & 196 & 336 \\
\end{tabular}
    \end{center}
    This is a particular case of Proposition \ref{prop:hreggrass}, which gives a general formula for the Hilbert regularity of Grassmannians.
\end{example}

\begin{prop} \label{prop:hreggrass}
Let ${\rm Gr}(k,m)$ be the Grassmannian of $(k-1)$-planes in $\mathbb{P}_K^{m-1}$ in its Pl\"ucker embedding. We have ${\rm HReg}(K[{\rm Gr}(k,m)]) = -m+1$.
\end{prop}
\begin{proof}
 By Equation (8.4) in \cite{mukai2003introduction}, for $k = 2$, i.e.~$X = {\rm Gr}(2,m) \subset \mathbb{P}^{\binom{m}{2}-1}$, we have 
\begin{equation} \label{eq:HPlines}
{\rm HP}_{K[{\rm Gr}(2,m)]}(t)  \, = \, \frac{(t+1)(t+2)^2 \cdots (t+m-2)^2(t+m-1)}{(m-2)!(m-1)!}.
\end{equation}
It easily follows from \cite[Proposition 8.4]{mukai2003introduction} that this equals the Hilbert function for all $t \geq 0$, and since ${\rm HF}_{K[X]}(t) = 0$ for $t < 0$, we conclude ${\rm HReg}(K[{\rm Gr}(2,m)]) = -m + 1$. The formula \eqref{eq:HPlines}, together with the rest of the argument, generalizes for higher $k$: 
\begin{equation} \label{eq:HPgrass}
    {\rm HP}_{K[{\rm Gr}(k,m)]}(t)  \, = \,  \frac{1!\,2!\, \cdots (k-1)!}{(m-k)! \cdots (m-1)!} \cdot   \prod_{i=1}^k (t+i)(t+i+1) \cdots (t+i+m-k-1). 
\end{equation}
This agrees with the Hilbert function for $t \geq -m +1$. In particular, it vanishes for $t = -m+1, \ldots, -1$. This was first discovered by Hodge \cite{hodge1942note}.
\end{proof}

\begin{cor} \label{cor:grass}
    In the situation of Theorem \ref{thm:regfromhreg}, if $X = {\rm Gr}(k,m) \subset \mathbb{P}_K^{\binom{m}{k}-1}$, we have $n = k(m-k)$ and $\{ d_{[n]} + w :\, w \geq -m + 1\} \subset {\rm Reg}(I)$.
\end{cor}

The Hilbert regularity ${\rm HReg}(K[X])$ of an $n$-dimensional projective variety $X$ can be computed from its \emph{Hilbert series} ${\rm HS}_X$. By \cite[Corollary 4.1.8]{bruns1998cohen}, there is a unique Laurent polynomial $P_X(u) = c_a u^a + c_{a+1} u^{a+1} +  \cdots + c_b u^b$ with $c_b \neq 0$ such that $P_X(1) \neq 0$ and  
\[ {\rm HS}_X(u) \,= \, \sum_{t = 0}^\infty {\rm HF}_{K[X]}(t) \, u^t \, = \, \frac{P_X(u)}{(1-u)^{n+1}}. \]
The Hilbert regularity is read from the largest exponent $b$ of $P_X$ \cite[Proposition 4.1.12]{bruns1998cohen}.
\begin{thm} \label{thm:brunsherzog}
    The Hilbert regularity ${\rm HReg}(K[X])$ of an $n$-dimensional projective variety $X$ is $b - n$, where $b$ is the largest exponent appearing in $P_X(u)$.
\end{thm}
Theorem \ref{thm:intro} is a straightforward corollary of Theorems \ref{thm:brunsherzog} and \ref{thm:regfromhreg}.

\begin{example}
    The Hilbert series of $X = {\rm Gr}(2,4)$ is ${\rm HS}_{{\rm Gr}(2,4)}(u) = (1 + u)/(1-u)^5$. 
\end{example}

\begin{example} \label{ex:hregduffing}
    We use \texttt{Macaulay2} to compute the Hilbert series of the surface $X\subset \mathbb{P}^4$ defined by \eqref{eq:implicitduffing}: ${\rm HS}_X(u) = (1+2u+2u^2)/(1-u)^3$. By Theorem \ref{thm:brunsherzog}, ${\rm HReg}(K[X]) = 2 - 2 = 0$. If $I = \langle F_1, F_2 \rangle \in K[X]_1$ defines finitely many points on $X$, Theorem \ref{thm:regfromhreg} predicts $\{ d \in \mathbb{Z} \, : \, d \geq 2 \} \subset {\rm Reg}(I)$. This is confirmed by Example \ref{ex:regduffing}. 
\end{example}

\section{Computational examples} \label{sec:6}

Theorem \ref{thm:eigenvaluethm2} suggests an algorithm for computing the coordinates of $z_i \in X \subset \mathbb{P}^\ell$ from the matrices $(N_h)_{|B}^{-1} (N_p)_{|B}$. This section illustrates that algorithm via our proof-of-concept implementation in Julia, which uses the package \texttt{Oscar.jl} (v0.12.1) \cite{OSCAR}. The code is available at \url{https://mathrepo.mis.mpg.de/KhovanskiiSolving}. Throughout the section, $I \subset K[X]$ defines a reduced subscheme $V_X(I)$ of $X$, consisting of $\delta$ points.

We will use the equations \eqref{eq:f1f2} from Example \ref{ex:duffing2} to illustrate the algorithm and our Julia functions $(n=2,\ell=4,\delta=5)$. We first load \texttt{KhovanskiiSolving.jl}, which is the Julia package accompanying this paper, and \texttt{Oscar.jl}. We also fix the field $K = \mathbb{Q}$, create the ring $K[t_1,t_2]$ and define the parameterization $\phi$ from Example \ref{ex:duffing}.
\begin{minted}{julia}
using KhovanskiiSolving, Oscar
K = QQ; R, (t1,t2) = PolynomialRing(K,["t1";"t2"])
φ = [t1^0; t1; t2; t1*(t1^2+t2^2); t2*(t1^2+t2^2)]
\end{minted}
Here the first entry of \texttt{φ} uses $t_1^{\, 0}$ instead of $1$ to avoid type issues. 
Our polynomials are
\begin{minted}{julia}
f1 = φ[1] + 3*φ[2] + 5*φ[3] + 7*φ[4]
f2 = 11*φ[1] + 13*φ[2] + 17*φ[3] + 19*φ[5]
\end{minted}
It was explained in Section \ref{sec:4} that our algorithm works with the matrix $M_X(d+e)$, where $d, d+e \in {\rm Reg}(I)$ for $I = \langle F_1, F_2 \rangle \subset K[X]$. Here $F_1, F_2$ are as in \eqref{eq:F1F2}. Our current implementation assumes $e = 1$: the input \texttt{dreg} should be such that \texttt{dreg} and \texttt{dreg-1} are in ${\rm Reg}(I)$. We have seen in Examples \ref{ex:regduffing} and \ref{ex:hregduffing} that we can take \texttt{dreg = 3}. Other inputs are the degrees $d_i$ of the equations, i.e., the numbers $d_i$ in \eqref{eq:homsystem}, and a list of leading exponents of the parameterizing functions in \texttt{φ}. 
\begin{minted}{julia}
dreg = 3; degs_f = [1;1]; leadexps = [leadexp(h,[-1;0]) for h in φ]
\end{minted}
Here the function \texttt{leadexp} computes the leading exponent of $h$ with respect to the weight vector $\omega = (0,-1)$. This corresponds to the weights in Example \ref{ex:duffingkhov}. We have now gathered all input needed for our function \texttt{get\_KM}, which constructs the Khovanskii-Macaulay matrix $M_X(d)$ for $d = 3 =$ \texttt{dreg}:
\begin{minted}{julia}
KM = get_KM([f1;f2],dreg,degs_f,φ,K,(t1,t2),leadexps)   
\end{minted}
Its columns are indexed by a monomial basis of $K[X]_3$. Their exponents are \texttt{αs} in 
\begin{minted}{julia}
bs, αs = get_basis(φ,dreg,K,(t1,t2),leadexps)
\end{minted}
There are ${\rm HF}_{K[X]}(3) = 28$ monomials in this basis. The output \texttt{bs} contains the corresponding basis elements of the degree $3$ piece of the algebra $K[t_0 \phi_0, \ldots, t_0 \phi_4]$ (omitting the factor $t_0^3$). The correspondence between \texttt{bs} and \texttt{αs} is as in \eqref{eq:xtot}. 

To compute $M_X(d)$ the function \texttt{get\_KM} needs to compute the expansions \eqref{eq:expandmac}. In the current implementation, this is done via interpolation, i.e., we find the coefficients $M_X(d)_{(i,\gamma),\beta}$ by imposing that the resulting polynomial agrees with $b_{d-d_i,\gamma} \cdot f_i$ in a number of points large enough to define them uniquely. On a case-by-case basis, one can investigate if this can be made more efficient, possibly exploiting some extra structure in the algebra $K[t_0 \phi_0, \ldots, t_0 \phi_4]$. We do not elaborate on this here. 

The Khovanskii-Macaulay matrix in our example has size $(2 \cdot {\rm HF}_{K[X]}(2) ) \times {\rm HF}_{K[X]}(3)$, i.e., $28 \times 28$. Though it is not the case here,  $M_X(d)$ for $d =$ \texttt{dreg} often has (many) more rows than columns. Since we are only interested in its kernel $N_X(d)$, we might as well work with a submatrix consisting of ${\rm rank}\, M_X(d)$ linearly independent rows. This is more efficient, so we gave \texttt{get\_KM} the optional boolean input \texttt{reduce}. To compute a reduced Macaulay matrix, simply add ``\texttt{...,leadexps;reduce = true)}'' to the command \texttt{get\_KM} above. With that option, the size of \texttt{KM} is $23 \times 28$.

We now use \texttt{KM} to compute the matrices $(N_h)_{|B}^{-1} (N_p)_{|B}$ from Theorem \ref{thm:eigenvaluethm2}:
\begin{minted}{julia}
Mul, c = get_multiplication_matrices(KM,dreg,φ,K,(t1,t2),leadexps)
\end{minted}
Notice that \texttt{KM} is one of the inputs. The list \texttt{Mul} contains $\ell + 1$ matrices $(N_h)_{|B}^{-1} (N_p)_{|B}$, where $h = c_0x_0 + \cdots + c_\ell x_\ell$ is the same for all these matrices, and $p = x_j, j = 0, \ldots, \ell$. The output \texttt{c} contains the coefficients of $h$, which are chosen randomly internally. Constructing these matrices involves only basic linear algebra. The reader can find the details in our publicly available code. The name \emph{multiplication operators} comes from the interpretation of the matrices in \texttt{Mul} as multiplication endomorphisms in the $\delta$-dimensional vector space $(K[X]/I)_{\mathtt{dreg}-1}$, see \cite[Section 4.5]{telen2020thesis}.

We discuss how to compute the coordinates of the points $V_X(I) = \{z_1, \ldots, z_\delta \} \subset X$ from the matrices in \texttt{Mul}. Since the coefficients \texttt{c} are chosen randomly, we may assume
$V_X(I) \cap \{ h = 0 \} = \emptyset$. The coordinate $x_j/h$ of all points $V_X(I)$ is read from the eigenvalues of the $j$-th matrix $\mathtt{Mul}_j = (N_{x_0})_{|B}^{-1} (N_{x_j})_{|B}$ in \texttt{Mul} (Theorem \ref{thm:eigenvaluethm2}). Moreover, for $j' \neq j$, it is easy to determine which eigenvalues of $\mathtt{Mul}_j$ and $\mathtt{Mul}_{j'}$ belong to the same solution. This claim is justified by the proof of Theorem \ref{thm:eigenvaluethm2}, where we saw that all matrices $(N_h)_{|B}^{-1}(N_p)_{|B}$ have the same eigenvectors. Suppose $Z$ is an eigenvector matrix. The diagonal elements of $Z \cdot (N_{x_0})_{|B}^{-1}(N_{x_j})_{|B} \cdot Z^{-1}$ correspond to the solutions $z_1, \ldots, z_\delta$ in an ordering that is fixed by the ordering of the eigenvectors. The~function \texttt{get\_solutions} of \texttt{SolvingOnParameterizedVarieties} implements~this procedure: 
\begin{minted}{julia}
sols = get_solutions(Mul)
\end{minted}
This works when $K = \mathbb{Q}$. The result \texttt{sol} is a $\delta \times (\ell+1)$ matrix whose rows are the $\ell + 1$ homogeneous coordinates of the solutions in $V_X(I)$. 

Note that almost all these computations can be done over the ground field $K$. Only for the final eigenvalue computations we switch to the closure $\overline{K}$. In particular, if $h,p$ have coefficients in $K \subset \overline{K}$, then $(N_h)_{|B}^{-1} (N_p)_{|B}$ has entries in $K$. In this example,
\texttt{get\_solutions} returns floating point approximations of the solutions in $\mathbb{P}_{\overline{\mathbb{Q}}}^5$.

All of the above steps are executed by the wrapper function \texttt{solve\_Khovanskii}:
\begin{minted}{julia}
sols = solve_Khovanskii([f1;f2],dreg,degs_f,φ,(t1,t2),leadexps)
\end{minted}

We illustrate some more functionalities of our code in two examples. The first one deals with solving equations on the del Pezzo surface from Example \ref{KB example}. In the second, $X$ is a \emph{Bott-Samelson threefold}. The next section is devoted to the case $X = {\rm Gr}(k,m)$. 

\begin{example}[Equations on a del Pezzo surface] \label{ex:delpezzocomp}
Here is a snippet of code which solves two random degree $d$ equations $f_1 = f_2 = 0$ on the del Pezzo surface $X$ from Example \ref{KB example}, defined over a prime field $K = \mathbb{F}_p$ of characteristic $p = 9716633$:
\begin{minted}{julia}
d = 2; p = 9716633; K = GF(p)
R, (t1,t2) = PolynomialRing(K, ["t1";"t2"])
φ = [t1-t2; t2^2-t2; t1*t2-t2; t1^2-t2; t1*t2^2-t2; t1^2*t2-t2]
S, x = PolynomialRing(K, ["x$i" for i = 1:length(φ)])
evs = collect(exponents(sum(x)^d)); mons = [prod(φ.^ev) for ev in evs]
c1 = rand(1:p,length(mons)); c2 = rand(1:p,length(mons))
f1 = (c1'*mons); f2 = (c2'*mons)
degs_f = [d;d]; dreg = sum(degs_f)+1; 
leadexps = [leadexp(h,[-2;-1]) for h in φ]
Mul = get_commuting_matrices([f1;f2],dreg,degs_f,φ,K,(t1,t2),leadexps)
\end{minted}
Since eigenvalue computations over finite fields are not supported in the version of \texttt{Oscar.jl} we use, we only compute the six multiplication matrices. The command \texttt{get\_commuting\_matrices} executes \texttt{get\_KM} and \texttt{get\_multiplication\_matrices}. The formula $\mathtt{dreg} = 2d+1$ is based on Theorem \ref{thm:regfromhreg}, after computing ${\rm HReg}(K[X]) = 0$. The number of solutions is $5d^2$, which is also the size of the six matrices in \texttt{Mul}. The size of the Khovanskii-Macaulay matrix $M_X(d)$ is ${\rm HF}_{K[X]}(d)$, which equals the Ehrhart polynomial $5/2d^2 + 5/2d + 1$ of a pentagon (see Remark \ref{rem:ehrhart}). We executed this computation for $d$ up to $15$, for which it takes $1169$ seconds on a single thread of a 2.8 TB RAM machine using an Intel Xeon E7-8867 v3 processor working at 2.50 GHz.
\end{example}

\begin{example}[Equations on a Bott-Samelson variety]
We use our code to reproduce Example 10 in \cite{burr2020numerical}. This example considers a threefold $X \subset \mathbb{P}^7$ parameterized by 
\begin{minted}{julia}
K = QQ; R, (t1,t2,t3) = PolynomialRing(K, ["t1";"t2";"t3"])
φ = [t1^0; t1; t2; t3; t1*t3; t2*t3; t1*(t1*t3+t2); t2*(t1*t3+t2)]
\end{minted}
The entries of $\phi$ form a basis for the sections of an ample line bundle on a \emph{Bott-Samelson variety}. The threefold $X$ has degree 6, which is the number of solutions to \texttt{f = 0}, with
\begin{minted}{julia}
f = [1 1 1 1 1 1 1 1; 1 -2 3 -4 5 -6 7 -8; 2 3 5 7 11 13 17 19]*φ
\end{minted}
These are interpreted as three linear equations on $X$ ($d_1 = d_2 = d_3 = 1)$. A Gr\"obner basis for $I(X)$ is displayed in \cite[Example 10]{burr2020numerical}. From this, we compute in Macaulay2~that 
\[ {\rm HS}_X(u) \, = \, \frac{1 + 4u + u^2}{(1-u)^4}, \quad \text{ and thus } \quad  {\rm HReg}(K[X]) = -1\]
by Theorem \ref{thm:brunsherzog}. Relying on Theorem \ref{thm:regfromhreg}, we set ${\tt dreg} = d_1 + d_2 + d_3 = 3$:
\begin{minted}{julia}
degs_f = [1;1;1]; dreg = 3; leadexps = [leadexp(h,[0;-1;0]) for h in φ]
sols = solve_Khovanskii(f,dreg,degs_f,φ,(t1,t2,t3),leadexps)
sols = sols./sols[:,1]
\end{minted}
The last line normalizes the first projective coordinate in $\mathbb{P}^7$ to be 1. Among \texttt{sols} is the point
$(1:-0.689522:0.928435:-1.35986:0.937652:-1.26254:-1.28671:1.73254)$,
which is the solution displayed on page 8 of \cite{burr2020numerical}. 
\end{example}

\section{Solving equations on Grassmannians} \label{sec:7}

The Grassmannian ${\rm Gr}(k,m)$ is a prime example of a unirational, non-toric variety whose points naturally represent candidate solutions to equations from geometry and applications \cite{huber1998numerical}. Moreover, its Pl\"ucker embedding provides a Khovanskii basis (Example \ref{ex:grasskhov}) and we have an explicit formula for the degree of regularity of complete intersections (Theorem \ref{cor:grass}). Hence, Grassmannians provide an excellent test case for our algorithm. All computations in this section are done using the same hardware as in Example \ref{ex:delpezzocomp}.

In the notation of the Introduction, $X = {\rm Gr}(k,m)$ is an $n$-dimensional projective variety in $\mathbb{P}^{\ell}_K$, with $n = k(m-k)$ and $\ell = \binom{m}{k}-1$. 
The parameterizing functions $\phi_0, \ldots, \phi_\ell$ are the Pl\"ucker coordinates, i.e., the $k \times k$ minors, of the $k \times m$-matrix
\begin{equation} \label{eq:H} H= \begin{pmatrix}
   1& & & &\vrule&   t_1 & t_2 & \dots & t_{m-k-1} & t_{m-k}  \\  & 1 & & & \vrule& t_{m-k+1} & t_{m-k+2} & \cdots & t_{2m-2k-1} & t_{2m-2k} \\ & &\ddots& & \vrule & \vdots & \vdots & \cdots & \vdots & \vdots \\ & &  & 1&\vrule& t_{n-m+k+1}& t_{n-m+k+2} & \dots & t_{n-1} & t_n    
\end{pmatrix}. 
\end{equation}
In this section, we compute Pl\"ucker coordinates of solutions to homogeneous equations in these minors, as in \eqref{eq:system}. We start by solving Schubert problems, following the exposition in \cite[Section 9.3]{sottile2011real}.
A set of \emph{Schubert conditions} on a $(k-1)$-plane $H \in {\rm Gr}(k,m)$ consists of two pieces of data. The first is a set $\alpha \in \binom{[m]}{k}$ consisting of $k$ indices $\{\alpha_1, \ldots, \alpha_k \} \subset [m]=\{1,\dots,m\}$, ordered such that $\alpha_1 < \alpha_2 < \cdots < \alpha_k$. The second is a complete flag $F_\bullet$ in $\mathbb{P}^{m-1}$. That is, $F_\bullet$ is a sequence 
\[ F_\bullet: \,  \ F_1 \, \subset \,  F_2 \, \subset \, \dots \, \subset \,  F_m \, = \, \mathbb{P}^{m-1} \]
of linear subspaces, where $F_i$ has dimension $i-1$.
Such a flag is represented by an $m\times m$ matrix whose $i$ rows span $F_i$ projectively. 
The data $(\alpha, F_\bullet)$ define the \emph{Schubert~variety}
\[ X_\alpha F_\bullet \,  = \,  \{ H\in {\rm Gr}(k,m) \, : \, {\rm dim}_K(H\cap F_{\alpha_i})\geq i-1 \ \, {\rm for} \,  \ i=1,\dots,k \}. \]
This subvariety of ${\rm Gr}(k,m)$ has dimension $D(\alpha)=\alpha_1-1+\alpha_2-2+\dots+\alpha_k-k$. 
Its equations are obtained as follows. We abuse notation slightly and write $F_{\alpha_i}$ for an $\alpha_i \times m$-matrix whose rows span $F_{\alpha_i}$.
The condition ${\rm dim}_K(H\cap F_{\alpha_i})\geq i-1 $ is equivalent to $ {\rm dim}_K(H+F_{\alpha_i})\leq k+\alpha_i-i$. Hence, the defining equations of $X_\alpha F_\bullet$ are given by the vanishing of the $(k + \alpha_i - i +1)$-minors of the $(k+\alpha_i) \times m$-matrix $\begin{pmatrix}
    H \\ F_{\alpha_i}
\end{pmatrix}$, for $i = 1, \ldots, k$. Here, values of $i$ for which such minors do not exist, i.e. $k + \alpha_i - i + 1 > \min(k+\alpha_i,m)$, are skipped. Indeed, for such $i$, the condition $\dim_K(H \cap F_{\alpha_i}) \geq i-1$ is trivial. Note that all these minors are \emph{linear} combinations of Pl\"ucker coordinates, i.e., $d_i = 1$.

A \emph{Schubert problem} is given by a list of Schubert conditions $(\alpha^1, F^1_\bullet),\dots,(\alpha^c, F^c_\bullet)$ such that $\sum_{j=1}^c (n-D(\alpha^j))=n$. Here $c$ can be any positive integer, and the flags $F^j_\bullet$ are assumed to be in general position. The intersection of these Schubert varieties 
\begin{equation} \label{Schuber variety}
    X_{\alpha^1}F^1_\bullet \, \cap \,  X_{\alpha^2}F^2_\bullet \, \cap \, \cdots \, \cap \,  X_{\alpha^c}F^c_\bullet
\end{equation}
consists of finitely many points in ${\rm Gr}(k,m)$. Our task is to compute these points. 

\begin{example}[$k=3, m = 6$] \label{ex Schubert calculus}
    We define a Schubert problem in ${\rm Gr}(3,6)$ by setting $\alpha^1 = \alpha^2 = \alpha^3 =(2,4,6)$ and picking three general flags $F^1_\bullet, F^2_\bullet, F^3_\bullet$. We have 
    \[ (9-6) + (9-6) + (9-6) \, = \, 9 \, = \, n. \]
The three Schubert varieties $X_{\alpha^j}F^j_\bullet $ each have codimension $3$. Their equations are the five $5 \times 5$-minors of the $5 \times 6$ matrix $\binom{H}{F^j_4}$, and 
the seven $6 \times 6$-minors of the $7 \times 6$-matrix $\binom{H}{F^j_4}$. 
This gives 13 equations for each $j$, hence a total of 39  equations on the $9$-dimensional space ${\rm Gr}(3,6)$. Among these, only 18 are linearly independent, so we can delete 21 equations in an easy pre-processing step. The remaining equations cut out the two $2$-planes in $\mathbb{P}_{\overline{K}}^{5}$ that interact in the prescribed way with our three flags.
\begin{minted}{julia}
k = 3; m = 6; n = k*(m-k); K = QQ; A = [[2,4,6], [2,4,6], [2,4,6]];
R, φ, t, M = plueckercoordinates(k,m,K);
Flags, F, leadexps, degs_F = equationsSchubertVariety(A,k,m,K); dreg = 2;
sol = solve_Khovanskii(F,dreg,degs_F,φ,t,leadexps)
\end{minted}
A different Schubert problem in ${\rm Gr}(3,6)$ is given by $\alpha^1 = \cdots = \alpha^9 = (3,5,6)$ and nine general flags. Here each Schubert variety $X_{\alpha^j}F^j_\bullet \subset {\rm Gr}(3,6)$ is a hypersurface defined by a single linear equation in the Pl\"ucker coordinates. The number of points in $\bigcap_{j = 1}^9 X_{\alpha^j} F^j_\bullet$ is the degree of ${\rm Gr}(3,6)$ in its Pl\"ucker embedding, which is 42.

We end this example by solving five Schubert problems on ${\rm Gr}(3,6)$ using our Julia package \texttt{KhovanskiiSolving.jl}. We use all possible combinations of the parameter vectors $\alpha^1 = (3,5,6)$ and $\alpha^2 = (2,5,6)$. The results are shown in Table \ref{tab:schubert_problem_Gr_K(3,6)}. The solutions are computed using the KM matrix $M_X(d_{\rm reg})$, where the value of $d_{\rm reg}$ is determined experimentally if the number of equations is greater than 9. In these cases, the results from Section \ref{sec:4} do not apply. The parameters $t_{K}$ indicate the computation time with ground field $K$. As expected, computations over a finite field are more efficient. The column ${\rm HF}_{K[X]}(d_{\rm reg})$ shows the size of the KM matrix used in the computation. For $9 \times \alpha^1$, this equals $14112$. That computation did not terminate for $K= \mathbb{Q}$ within reasonable time. We chose this example to demonstrate the limits of our implementation.
\end{example}

\begin{table}[h]
\small
    \centering
    \begin{tabular}{c|cccccc}
Schubert  cond. & \# solutions & $d_{\rm reg}$ & ${\rm HF}_{K[X]}(d_{\rm reg})$ & \# equations & 
$t_\mathbb{Q}$ & $t_{\mathbb{F}_{9716633}}$\\ \hline
$9\times \alpha^1$ & 42 & 5 & 14112 & 9 & $\times$  & $15635s$  \\ 
$7\times \alpha^1+1\times \alpha^2$  & 21  & 4 & 4116 & 11 & $12296s$ & $1074s$  \\ 
$5\times \alpha^1+2\times\alpha^2$ & 11 & 3 & 980 & 13 & $598s$ & $51s$  \\ 
$3\times \alpha^1+3\times\alpha^2$ & 6 & 3 & 980 & 15 & $407s$ &$40s$ \\ 
$1\times \alpha^1+4\times\alpha^2$ & 3 & 2 & 175 & 17 & $11s$ & $2s$  \\ 

\end{tabular}
    \caption{Computational results for Schubert problems on ${\rm Gr}(3,6)$.}
    \label{tab:schubert_problem_Gr_K(3,6)}
\end{table}

If $K = \mathbb{C}$, there are particular real flags whose Schubert problems have only real solutions. Let $\gamma(s)=(1,s,s^2,\dots,s^{m-1})\in \mathbb{R}^m$ for $s\in \mathbb{R}$. The \emph{osculating flag} $F_\bullet(s)$ at the point $s\in\C$ is the flag whose $i$-dimensional plane is spanned by $\gamma(s), \gamma'(s), \ldots, \gamma^{(i-1)}(s)$. 
For any $\alpha^1, \ldots, \alpha^c$ defining a Schubert problem and every choice of distinct points $s_j\in \mathbb{R}$, $j=1,\dots,c$, the intersection of the corresponding Schubert varieties
    \[ X_{\alpha^1}F_\bullet(s_1) \, \cap \, X_{\alpha^2}F_\bullet(s_2) \, \cap \,  \dots \, \cap \,  X_{\alpha^c}F_\bullet(s_c)
    \] 
in ${\rm Gr}_{\mathbb{C}}(k,m)$ is finite and transverse, and all its points are real, see \cite[Theorem 9.13]{sottile2011real}. We demonstrate this for the 6-dimensional Grassmannian ${\rm Gr}(2,5)$. 

\begin{example}[$k = 2, m = 5$]
The $2$-planes of the osculating flags $F_\bullet(\pm i)$, $i = 1, 2, 3$ in $\mathbb{P}^4$ are represented by the following six $3 \times 5$-matrices :
\[
\begin{matrix}
\begin{pmatrix}
    1 & \pm 3 & 9 & \pm27 &81\\
    0 & 1 & \pm6 & 27 & \pm 108 \\
    0 & 0 & 2 & \pm 18 & 108
\end{pmatrix}, &
\begin{pmatrix}
    1 & \pm 2 & 4 & \pm 8 &16\\
    0 & 1 & \pm 4 & 12 & \pm3 2\\
    0 & 0 & 2 & \pm 12 & 48
\end{pmatrix}, &
\begin{pmatrix}
    1 & \pm 1 & 1 & \pm 1 &1\\
    0 & 1 & \pm 2 & 3 & \pm4\\
    0 & 0 & 2 & \pm 6 & 12
\end{pmatrix}.
\end{matrix}
\]
The Schubert condition $(3,5)$ imposes that a line $H$ touches each of these $2$-planes. This is equivalent to imposing that the full rank $2 \times 5$-matrix $H$ completes all six of the above matrices to a rank deficient $5 \times 5$-matrix. The resulting linear system in Pl\"ucker coordinates has 5 solutions on ${\rm Gr}_{\mathbb{C}}(2,5)$, and all of them are real. They can be found using \texttt{SolvingOnParameterizedVarieties.jl}. 
One of the solutions is 
\[ \begin{pmatrix}
    1 & 0 & 0 & 2.24227 & -16.3333 \\
    0 & 1 & 0 & -4.66667 & 17.9382
\end{pmatrix}. \qedhere \]
\end{example}

All Schubert problems above give rise to \emph{linear} equations on ${\rm Gr}(k,m)$. That is, the degrees $d_i$ in \eqref{eq:homsystem} are all equal to 1. We now switch to nonlinear equations, which we obtain by replacing the flags $F_\bullet$ by nonlinear objects. Let $Y \subset \mathbb{P}_K^{m-1}$ be an irreducible projective variety of dimension $m-k-1$. We consider the set of $(k-1)$-planes
\[ \{ H \in {\rm Gr}(k,m) \, : \, H \cap Y \neq \emptyset \}\]
intersecting $Y$ in at least one point. This is an irreducible hypersurface in ${\rm Gr}(k,m)$ \cite[Theorem 1.1]{dalbec1995introduction}. Its unique defining equation is called the \emph{Chow form} of $Y$. We denote it by ${\rm Ch}(Y) \in K[{\rm Gr}(k,m)]$. The Chow form can be computed from a set of equations of $Y$ using steps 0-4 in \cite[Section 3.1]{dalbec1995introduction}. In fact, for our purposes, it suffices to stop after step 3, which expresses ${\rm Ch}(Y)$ as a polynomial in the coordinates $t_i$ from \eqref{eq:H}. This polynomial is an element of the subalgebra of $K[t_1, \ldots, t_n]$ generated by the Pl\"ucker coordinates. Using these Chow forms for $f_i$ in \eqref{eq:system}, we can use our approach to answer questions of the type \emph{``Which $(k-1)$-planes intersect a given list of $(m-k-1)$-dimensional varieties $Y_1, \ldots, Y_n \subset \mathbb{P}^{m-1}$?''}. Here is an example for curves in three-space. 
\begin{example}[$k = 2, m = 4$]
    In this example, we work over $K = \mathbb{F}_{9716633}$. We generate curves of degree 4 in $\mathbb{P}^3$ by intersecting two generic conics in $K[x_0,x_1,x_2,x_3]_2$. Similarly, we construct degree 2 curves by intersecting a generic conic with a generic linear form. The Chow forms of such a quartic curves have degree 4 in $K[{\rm Gr}(2,4)]$, while that of the quadratic curve has degree 2. For $q = 0, 1, 2, 3, 4$, we build a system of $4$ equations on ${\rm Gr}(2,4)$ as follows. For $f_1, \ldots, f_q$, we use the Chow forms of $q$ degree 4 curves. The remaining equations $f_{q+1}, \ldots, f_4$ are the Chow forms of $4-q$ degree 2 curves. By solving these equations, we answer the question \emph{``Which lines in $\mathbb{P}_K^3$ intersect $q$ given quartic curves and $4-q$ given quadratic curves?''.} By the results of Section \ref{sec:4}, the parameter $d_{\rm reg}$ is given by $2q+6$. The results are reported in Table \ref{tab:chowforms}.
    \begin{table}[h]
    \centering
    \small
    \begin{tabular}{c|cccccc}
$q$ & \# solutions & $d_{\rm reg}$ & ${\rm HF}_{K[X]}(d_{\rm reg})$ & \# equations & 
$t_\mathbb{Q}$ & $t_{\mathbb{F}_{9716633}}$\\ \hline
$0$ & 32 & 6  & 336  & 4 & $39s$ & $4s$ \\ 
$1$  & 64  & 8 & 825   & 4 & $711s$  & $36s$  \\ 
$2$ & 128 & 10 & 1716  & 4 & $6833s$  & $247s$  \\ 
$3$ & 256 & 12 & 3185  & 4 & $55319s$ & $1336s$ \\ 
$4$ & 512 & 14 & 5440  & 4 & $\times$  &  $5572s$ \\ 
\end{tabular}
    \caption{Computational results for intersecting Chow forms on ${\rm Gr}(2,4)$.}
    \label{tab:chowforms}
\end{table}
\end{example}

In conclusion, this article proposes to use Khovanskii-Macaulay matrices in computer algebra methods for solving polynomial equations. This is natural in the case where the equations lie in a subalgebra with a finite Khovanskii basis, which happens, for instance, for equations on Grassmannians. The matrices then represent graded pieces of the corresponding ideal in the coordinate ring of the unirational variety.
We have laid the first theoretical foundations, including an Eigenvalue Theorem (Theorem \ref{thm:eigenvaluethm2}) and a more generally applicable regularity result (Theorem \ref{thm:regfromhreg}). Our Julia implementation and the examples in Sections \ref{sec:5} and \ref{sec:6} provide evidence for the effectiveness of this strategy. However, this first implementation has its limits. Directions for future research include an optimized implementation, the development of these methods in finite precision arithmetic for efficient and robust numerical computations, and establishing determinantal formulas for resultants on $X$ using Khovanskii-Macaulay matrices.

\paragraph{Acknowledgements.}
We would like to thank Elisa Gorla, Frank Sottile and Bernd Sturmfels for helpful discussions. Simon Telen was supported by a Veni grant from the Netherlands Organisation for Scientific Research (NWO).

\small
\bibliographystyle{abbrv}
\bibliography{references.bib}

\medskip 

\noindent{\bf Authors' addresses:}
\medskip

\noindent Barbara Betti, MPI-MiS Leipzig \hfill{\tt barbara.betti@mis.mpg.de}

\noindent Marta Panizzut, MPI-MiS Leipzig \hfill{\tt marta.panizzut@mis.mpg.de}

\noindent Simon Telen, MPI-MiS Leipzig and CWI Amsterdam
\hfill {\tt simon.telen@mis.mpg.de}

\end{document}